\documentclass[a4paper,12pt]{article}

\usepackage{amssymb}
\usepackage{amsfonts}
\usepackage{amsmath}
\usepackage{amsthm}
\usepackage{graphicx, subfigure, psfrag}

\def\es{\emptyset}
\def\eps{\varepsilon}

\def\Z{\mathbb{Z}}
\def\R{\mathbb{R}}

\def\dist{\mathrm{dist}}
\def\desc{\mathrm{desc}}
\def\conn{\mathrm{conn}}
\def\diam{\mathrm{diam}}
\def\cC{\mathcal{C}}
\def\bC{\mathbf{C}}
\def\cD{\mathcal{D}}
\def\bD{\mathbf{D}}

\def\sfE{\mathsf{E}}

\def\fL{\mathfrak{L}}
\def\fM{\mathfrak{M}}
\def\cR{\mathcal{R}}

\def\cT{\mathcal{T}}

\def\LE{\mathrm{LE}}

\def\E{\mathbf{E}}

\def\P{\mathbf{P}}
\def\E{\mathbf{E}}
\def\WSF{\mathsf{WSF}}
\def\FSF{\mathsf{FSF}}
\def\UST{\mathsf{UST}}

\def\tail{\mathrm{tail}}
\def\head{\mathrm{head}}
\def\Vrt{\mathrm{Vert}}

\newcommand{\eqn}[2]{\begin{equation}\label{#1}#2\end{equation}}
\newcommand{\eqnst}[1]{\begin{equation*}#1\end{equation*}}
\newcommand{\eqnspl}[2]{\begin{equation}\begin{split}\label{#1}%
   #2\end{split}\end{equation}}
\newcommand{\eqnsplst}[1]{\begin{equation*}\begin{split}%
   #1\end{split}\end{equation*}}

\theoremstyle{plain}
\newtheorem{theorem}{Theorem}
\newtheorem{lemma}[theorem]{Lemma}
\newtheorem{proposition}[theorem]{Proposition}

\theoremstyle{definition}
\newtheorem{definition}{Definition}

\theoremstyle{remark}
\newtheorem{remark}{Remark}

\newtheorem{open}{Open Question}

\begin{document}

\title{Anchored burning bijections on finite and infinite graphs}
\author{
Samuel L. Gamlin\thanks{Department of Mathematical Sciences, University of Bath,
Claverton Down,
Bath, BA2 7AY,
United Kingdom,
Email: {\tt S.Gamlin@bath.ac.uk}} \and
Antal A. J\'{a}rai\thanks{Department of Mathematical Sciences, University of Bath,
Claverton Down,
Bath, BA2 7AY,
United Kingdom,
Email: {\tt A.Jarai@bath.ac.uk}}
}

\maketitle

\abstract{Let $G$ be an infinite graph such that each tree in the wired uniform spanning forest on $G$ has one end almost surely. On such graphs $G$, we give a family of continuous, measure preserving, almost one-to-one mappings from the wired spanning forest on $G$ to recurrent sandpiles on $G$, that we call anchored burning bijections. In the special case of $\Z^d$, $d \ge 2$, we show how the anchored bijection, combined with Wilson's stacks of arrows construction, as well
as other known results on spanning trees, yields a power law upper bound on the rate of convergence to the sandpile measure along any exhaustion of $\Z^d$. We discuss some open problems related to these findings.}

\medbreak

{\bf Key words:} abelian sandpile, uniform spanning tree, loop-erased random walk, Wilson's algorithm, burning algorithm, wired spanning forest.

\section{Introduction}

The abelian sandpile model is a stochastic particle model defined on a graph by a cellular automaton. Also known as the chip-firing game \cite{Holroyd(08)}, this model has shown interesting connections to a variety of different areas. In \cite{Bak(87)} the idea of self-organized criticality was introduced and
the sandpile model was used as a simple example of the properties sought. Dhar \cite{Dhar(90)} was the first to study the model in its own right, realising that many of its properties can be computed exactly, and
hence it has the capacity to demonstrate important underlying principles of self-organized criticality. See the surveys
\cite{Dhar(06),Redig(06),Jarai(14)} for background.

Let $G = (V \cup \{ s \}, \sfE)$ be a finite, connected multigraph, with a
distinguished vertex $s$, called the sink. A \emph{sandpile configuration}, that we usually denote by $\eta$, consists of assigning an integer
number of particles $\eta(v) \in \{ 0, 1, 2, \ldots\}$ to every non-sink vertex $v \in V$, The sandpile $\eta$ is \emph{stable}, if $\eta(v) \in \{ 0, 1, \ldots, \deg_G(v)-1 \}$,
for all $v \in V$, where $\deg_G(v)$ is the degree of the vertex $v$ in the graph $G$.

The dynamics of the model consist of two ingredients. The first is called \emph{toppling}. This occurs when a vertex has at least as many particles as its degree. For such a vertex $v$, its height is reduced by its degree and one particle is sent along every edge incident with $v$ (i.e.~vertices with multiple edges connecting them to $v$ receive more than one particle). Particles reaching the sink are lost
(i.e.~we do not keep track of them). The toppling of $v$ is summarised by the mapping $\eta(w) \to \eta(w) - \Delta_{v,w}$, $w \in V$, using the Laplacian matrix
\begin{align*}
\Delta_{v,w}=
\begin{cases}
\deg_G(w) & \text{ if $w = v$;} \\
-a_{v,w}  & \text{ if $w \not= v$;}
\end{cases}
\end{align*}
where $a_{v,w}$ is the number of edges connecting vertices $v, w \in V$.
Starting from any sandpile $\eta$, carrying out all possible topplings (in any order) yields a unique stable sandpile $\eta^\circ$ \cite{Dhar(90)}.

The second ingredient of the model is \emph{particle additions}. Given a stable sandpile $\eta$,
we add a particle at a randomly chosen vertex $v \in V$, and then stabilize via
topplings, if necessary. Successive particle additions yield a Markov chain on the set of stable sandpiles. We denote the set of recurrent states of this Markov chain by $\cR_G$, and by $\nu_G$ the unique stationary distribution, that is the uniform distribution on $\cR_G$ \cite{Dhar(90)}. The following combinatorial characterization of $\cR_G$ follows from \cite{Dhar(90),MD(92)} (see also \cite{Holroyd(08)}):
\eqnst
{ \cR_G
  = \left\{ \eta \in \prod_{x \in V} \{ 0, \ldots, \deg_G(x)-1 \} :
    \text{$\eta$ is ample for all $\es \not= F \subset V$} \right\}. }
Here $\eta$ is called \emph{ample for $F$}, if there exists $x \in F$
such that $\eta(x) \ge \deg_F(x)$.

An important tool for investigating sandpile configurations is the burning algorithm of Dhar \cite{Dhar(90)}, that we now describe.
Given $\eta \in \cR_G$, at time $0$ we declare the sink to be ``burnt''. Following this, we successively ``burn'' vertices
where $\eta(x)$ is at least as much as the number of edges leading from $x$ to any unburnt neighbours. More precisely, we set
\eqnsplst
{ B_0
  := \{ s \}, \qquad\qquad
  U_0
  := V, }
and for $j \ge 1$ we inductively set:
\eqnsplst
{ B_j
  := \left\{ v \in U_{j-1} : \eta(v) \ge \deg_{U_{j-1}}(v) \right\}, \qquad\qquad
  U_j
  := U_{j-1} \setminus B_j. }
Here $B_j$ (resp.~$U_j$), are the sets of vertices burnt, 
(resp.~unburnt), at time $j$.
Since $\eta$ is ample for any non-empty $U_{j-1}$, we have $U_j = \es$ eventually, 
at which time the algorithm terminates.

Majumdar and Dhar \cite{MD(92)}, following the above burning algorithm, 
constructed a bijection $\varphi_G : \cR_G \to \cT_G$, 
where $\cT_G$ is the set of spanning trees of $G$. 
The map $\varphi_G : \eta \mapsto t$, that arises as a special case 
of the bijections introduced in Section \ref{sec:anchored-bijection},
can be defined as follows. Fix for each $v \in V$ an ordering $\prec_v$ 
of the oriented edges $\{ f : \tail(f) = v \}$. If $v \in B_j$, let 
\eqnspl{e:mvFv}
{ m_v
  &:= \Big| \Big\{ f : \tail(f) = v,\, \head(f) \in \bigcup_{j' < j} B_{j'} \Big\} \Big|, \\
  F_v
  &:= \left\{ e : \tail(e) = v,\, \head(e) \in B_{j-1} \right\}. }
Due to the burning rule, we have
\eqn{e:etav}
{ \eta(v)
  = \deg(v) - m_v + \ell \quad \text{ for some $0 \le \ell < |F_v|$.} }
With $\ell$ as above, let $e_v \in F_v$ be that edge $e$ such that $| \{ f \in F_v : f \prec_v e \} | = \ell$. 
Then we place, for each $v \in V$, the directed edge $e_v$ in $t$, and forget the orientation.
Observe that the burning time of a vertex $v \in V$ equals $\dist_t(v,s)$, where
$\dist_t(\cdot,\cdot)$ is graph distance in $t$.

The image of $\nu_G$ under $\varphi_G$
is the uniform spanning tree measure $\UST_G$, i.e.~the uniform distribution on $\cT_G$. 
The burning bijection has been very fruitful in proving things about the sandpile model; 
see e.g.~\cite{Pr(94),JPR(06),Athreya(04),JR(08)}. It is natural to look for an extension 
of the burning bijection to infinite graphs, and this leads to some highly non-trivial questions. 
The main difficulty is that on finite graphs the burning algorithm starts from the sink, 
so the analogous process on infinite graphs should start from infinity. Our paper will be 
concerned with a particular way of overcoming this problem. However, as we outline below, 
some very natural questions remain open.

From now on, let $G = (V,\sfE)$ be a locally finite, connected, infinite graph. Given an 
exhaustion by finite subgraphs: $V_1 \subset V_2 \subset \dots \subset V$,
$\cup_{n=1}^\infty V_n = V$, let $G_n = (V_n \cup \{ s \}, \sfE_n)$ denote the 
\emph{wired graph} obtained by identifying the vertices in $V \setminus V_n$, 
that becomes the sink $s$, and removing loop-edges at $s$. Note that there is a natural 
identification between $\sfE_n$ and those edges in $\sfE$ that have an endvertex in $V_n$. 
We denote by $\WSF$ the weak limit of the measures $\UST_{G_n}$ \cite{LPbook}, called the 
\emph{wired uniform spanning forest measure} on $G$. It is well-known, and easy to see, 
that $\WSF$ concentrates on spanning subgraphs of $G$ all of whose components are 
infinite trees. We say that an infinite tree has \emph{one end}, if any two infinite
self-avoiding paths in the tree have a finite symmetric difference.
We will assume that $G$ satisfies the following condition:
\eqn{e:one-end}
{ \text{$\WSF$-a.s.~all components have one end.} }
While, in general, condition \eqref{e:one-end} is difficult to verify, it is known to hold on a large class of graphs, including $\Z^d$, $d \ge 2$; see \cite{Pem(91),BLPS(01),LMS(08),LPbook}.
We denote \eqnst
{ \cT
  := \left\{ \parbox{6.7cm}{spanning subgraphs of $G$ such that all
     components are infinite one-ended trees} \right\}. } The counterpart of $\cT$ for the sandpile model will be \eqnst
{ \cR
  := \left\{ \eta \in \prod_{x \in V} \{0, \dots, \deg_G(x)-1 \} :
     \text{$\eta$ is ample for all finite $\es \not= F \subset V$} \right\}, }
that we call the \emph{recurrent configurations on $G$}.

Athreya and J\'{a}rai \cite{Athreya(04)} considered the case of $\Z^d$, $d \ge 2$,
$V_n = [-n,n]^d \cap \Z^d$, and they showed that $\nu_{G_n}$ has a weak limit $\nu$
that concentrates on $\cR$. When $2 \le d \le 4$ the argument is particularly transparent. 
It was shown by Pemantle \cite{Pem(91)} that when $2 \le d \le 4$, the measure $\WSF$ 
concentrates on the set
\eqnst
{ \cT^\conn
  := \left\{ \text{one-ended spanning trees of $\Z^d$} \right\}
  \subset \cT. } 
In this case the limiting sandpile measure $\nu$ is exhibited as the
image of $\WSF$ under a map $\psi : \cT^\conn \to \cR$. 
Here $\psi$ is defined essentially by inverting the relationships 
\eqref{e:mvFv}--\eqref{e:etav}, that can be made sense of in $\Z^d$ 
for $t \in \cT^\conn$. Namely, fix $t \in \cT^\conn$ and $v \in \Z^d$.
Let $v^*$ denote the unique vertex such that all infinite paths starting at a
neighbour of $v$ pass through $v^*$, and $v^*$ is nearest to $v$ with respect
to $\dist_t$ (such vertex exists because $t \in \cT^\conn$). Orient all edges 
of $t$ towards infinity (this makes sense, because $t$ has one end). 
Let
\eqnsplst
{ m'_v
  &:= \Big| \Big\{ f : \tail(f) = v,\, \dist_t(\head(f),v^*) < \dist_t(v,v^*) \Big\} \Big|, \\
  F'_v
  &:= \left\{ e : \tail(e) = v,\, \dist_t(\head(e),v^*) = \dist_t(v,v^*)-1 \right\}. }
Enumerate $F'_v$ as $e_0 \prec_v \dots \prec_v e_{|F'_v|-1}$, and let 
$0 \le \ell' < |F'_v|$ be the unique index such that $e_{\ell'} \in t$.
Then we set
\eqnst
{ \psi(t)(v)
  := \eta(v)
 : = 2d - m'_v + \ell', \quad v \in \Z^d. }

It is not difficult to see that $\psi$ is continuous on $\cT^\conn$. 
(In a certain sense, $\psi$ is the limit of the inverse bijections 
$\psi_{G_n} := \varphi_{G_n}^{-1} : \cT_{G_n} \to \cR_{G_n}$.)
Moreover, $\psi$ is equivariant under translations of $\Z^d$, if the orderings
$\{ \prec_v : v \in \Z^d \}$ are chosen equivariant. It is tempting to 
conjecture that $\psi$ is almost 
one-to-one, i.e.~injective up to sets of measure $0$. We do not have a proof of this.

\begin{open}
\label{op:one-to-one}
Is $\psi$ almost one-to-one in the case of $\Z^d$, $2 \le d \le 4$?
\end{open}

When $d > 4$, it turned out to be necessary to add extra randomness to the $\WSF$ in order to construct $\nu$ \cite{Athreya(04)}, so there
is no natural mapping $\cT \to \cR$, a priori.

The main result of this paper is the construction of a family of measure preserving 
mappings between spanning forests and sandpiles that are almost one-to-one.
Our mappings can be constructed on general infinite graphs $G$ satisfying condition 
\eqref{e:one-end}, in particular, also on some non-transitive graphs. 
In this general setting, J\'{a}rai and Werning \cite{JW(12)} showed that $\nu_{G_n}$ 
converges weakly to a limit $\nu$, that is independent of the exhaustion.  
Our construction is a natural extension of the one in \cite{JW(12)}, that in turn was 
based on an observation of Majumdar and Dhar \cite{MD(91)} and Priezzhev \cite{Pr(94)}.
In general, when $G = (V,\sfE)$ is transitive, our mappings will \emph{not} be invariant 
under all graph automorphism.

\begin{definition}
\label{defn:anchor}
An \emph{anchor} is a sequence $\cD = \{ D_1, D_2, \ldots \}$ of finite subsets of vertices such that \begin{itemize}
\item[(i)] $D_1 \subset D_2 \subset \dots$ and $\cup_{k \ge 1} D_k = V$;
\item[(ii)] $D_k$ is simply connected for each $k \ge 1$, i.e.~all connected components of $V \setminus D_k$ are infinite. \end{itemize}
\end{definition}

In Section \ref{sec:anchored-bijection} we will associate to any anchor
$\cD$ a particular burning rule. That is, for any finite $\Lambda \subset V$ and configuration $\eta \in \cR_\Lambda$ we define burning times
$\tau^\cD(x,\Lambda; \eta)$, $x \in \Lambda$ in such a way that at each time
only vertices that are burnable in the sense of Dhar \cite{Dhar(90)} are burnt.
The advantage of our rule will be that it is easy to pass to the limit $\Lambda \uparrow V$, i.e.~we can define a consistent set of burning times
$\tau^\cD(x;\eta) \in \Z$ for $\nu$-a.e.~$\eta \in \cR$. The reason for requiring (ii) in Definition \ref{defn:anchor} is that for general
$D_k$, our burning rule will be identical if we replace $D_k$ by the
smallest simply connected set containing it.

\begin{theorem}
\label{thm:anchored-bijection}
Assume that the infinite graph $G$ satisfies condition \eqref{e:one-end}.
The burning rule arising from any anchor $\cD$ defines a continuous, measure preserving, injective map $\psi_\cD$ from $(\cT,\WSF)$ to $(\cR,\nu)$.
\end{theorem}

The precise meaning of ``defines'' will become clear in Section \ref{sec:anchored-bijection}, 
where we introduce the anchored bijection and the map $\psi_\cD$. Indeed, the anchor will serve to 
prescribe a ``preferred direction'' for the burning of configurations on $V$ starting from infinity. 
The following question complements Open Question \ref{op:one-to-one}.

\begin{open}
\label{op:d>4}
For $\Z^d$, $d > 4$, is there a continuous measure preserving map from $(\cT,\WSF)$ to $(\cR,\nu)$ that is equivariant with respect to translations?
\end{open}

Open Questions \ref{op:one-to-one} and \ref{op:d>4} are connected to a result of Schmidt and Verbitskiy \cite{SV(09)}. They constructed, for any $d \ge 2$, a family of $\Z^d$-equivariant continuous surjective mappings from $\cR$ onto the so called harmonic model, i.e.~functions from $\Z^d$ to the unit circle that are harmonic modulo $1$. The image of $\nu$ under their maps is the unique measure of maximum entropy of the harmonic model \cite[Theorem 5.9]{SV(09)}.

As an application of the anchored bijection, we show that combined with Wilson's stacks of arrows construction \cite{W(96)} it yields a coupling between $\nu_{G_n}$ and $\nu$ that we can analyze on $\Z^d$, $d \ge 2$. This leads to a power law upper bound on the rate of convergence of $\nu_{G_n}$ to $\nu$.

\begin{theorem}
\label{thm:power-law}
Let $d \ge 2$ and let $\Lambda \subset \Z^d$ be finite. Let $N$ be the radius of the largest
ball centred at the origin that is contained in $\Lambda$. There exists $\alpha = \alpha(d) > 0$
such that for any $1 \le k < N$ and any cylinder event $E$ depending only on the heights within distance $k$ of the origin we have
\eqn{e:power-law}
{ |\nu_\Lambda(E) - \nu(E)|
  \le C(k,d) N^{-\alpha}. }
\end{theorem}

The exponent $\alpha$ and the dependence on $k$ are explicit, although not optimal; see Theorem \ref{thrm:d>2} and Theorem \ref{thrm:d=2} for more detailed statements. Estimates analogous to \eqref{e:power-law}, but restricted to $d = 2, 3$, have been given in the context of the zero dissipation limit in the abelian avalanche model \cite{Jarai(11),JRS(14)}. We believe that our approach will lead to a significant simplification, and an extension to all $d \ge 2$, of the arguments of \cite{Jarai(11)}.

As mentioned earlier, we will define burning processes on both finite and infinite configurations in such a way that these behave well
with respect to taking limits. In particular, restricting an infinite recurrent
configuration to distinct large finite sets $\Lambda_1, \Lambda_2$, the anchored burning processes on $\Lambda_1$ and $\Lambda_2$ couple with high probability, in the following sense:
\eqn{e:burn-couple}
{ \lim_{\Lambda \uparrow V} \left[ \tau^\cD(x,\Lambda;\eta) - \tau^\cD(y,\Lambda;\eta) \right]
  = c(x,y;\eta). }
We do not know whether the same statement is true for Dhar's original burning algorithm, where at each step \emph{every} burnable vertex is burnt simultaneously.

\begin{open}
\label{op:burn-couple}
Let $\tau(x,k;\eta)$ denote the burning time of $x$ with respect to Dhar's original burning algorithm in the ball of radius $k$ centred at the origin
in $\Z^d$. Does the analogue of \eqref{e:burn-couple} hold for $\Z^d$, $2 \le d \le 4$, as $k \to \infty$?
\end{open}

If the answer is yes, this would imply an affirmative answer to Open Question \ref{op:one-to-one}. This is because the coupling defines a burning time from infinity (unique up to a time shift) and this can be used to define the inverse map.
Note that the arguments of \cite{Athreya(04)} show that the statement of
Open Question \ref{op:burn-couple} fails for $\Z^d$, $d > 4$.

We close this introduction by remarking that a certain analogue of the statement of Open Question \ref{op:burn-couple} holds on graphs of the form $G = G_0 \times \Z$,
with $G_0$ a finite connected graph. Indeed, with respect to the left-burnable measure studied by J\'{a}rai and Lyons \cite{JL(07)}, it is not difficult to construct a configuration on a ``triple of columns''
$G_0 \times \{ 1, 2, 3 \}$ that ``synchronizes'' burning from the left,
and hence coupling occurs. It was in fact by studying this case that we
arrived at the idea of anchored bijections.

\medbreak

The paper has the following structure. In Section \ref{sec:anchored-bijection} we define the anchored bijection in the finite case and then show how this extends to give a bijection in the infinite case. In Section \ref{sec:d>=3} we present the quantitative bounds on $\Z^d$ when $d \ge 3$. In Section \ref{sec:d=2} we give the bounds on $\Z^2$.
Throughout $C_d$ will stand for an unspecified positive constant dependent only on $d$, that we change without any further indication.

\section{Anchored bijections}
\label{sec:anchored-bijection}

Recall that $G = (V,\sfE)$ is a locally finite infinite graph satisfying \eqref{e:one-end};
we allow parallel edges. Let $\cD = \{ D_1, D_2, \dots \}$ be an anchor,
and let $D_0 := \es$. We call the set $E_k := D_k \setminus D_{k-1}$ the \emph{$k$-th shell}.
Given any finite $\Lambda \subset V$, we form the wired graph
$G_\Lambda = (\Lambda \cup \{ s \}, \sfE_\Lambda)$, and denote \eqnsplst
{ \cT_\Lambda
  &= \text{collection of spanning trees in $G_\Lambda$}, \\
  \cR_\Lambda
 &= \text{recurrent sandpiles in $G_\Lambda$}. }
We first define a bijection between $\cR_\Lambda$ and $\cT_\Lambda$ that is an extension of the one considered in \cite{JW(12)}.

\medbreak

\emph{Anchored bijection in finite $\Lambda$.}

Let $K = \max \{ k \ge 0 : D_k \subset \Lambda \}$. Fix $\eta \in \cR_\Lambda$.
Our definitions will depend on $\cD$, but we will not always indicate this in our notation.

\medbreak

{\bf Phase 1.} We apply the usual burning algorithm to $\eta$ with the
restriction that we do not allow any vertex of $D_K$ to burn. That is, we
define
\eqnsplst
{ B^{(1)}_0
  &:= \{ s \}, \\
  U^{(1)}_0
  &:= \Lambda, }
and for $j \ge 1$ we inductively set:
\eqnsplst
{ B^{(1)}_j
  &:= \left\{ v \in U^{(1)}_{j-1} \setminus D_K : \eta(v) \ge \deg_{U^{(1)}_{j-1}}(v) \right\}, \\
  U^{(1)}_j
  &:= U^{(1)}_{j-1} \setminus B^{(1)}_j. }
We have $B^{(1)}_j = \es$ eventually.
Note that there may be vertices in $\Lambda \setminus D_K$ that do not 
burn in Phase $1$. These vertices, together with the vertices in $D_K$,
will burn in later phases.

\medbreak

Assuming Phase $i-1$ has already been defined for some $2 \le i \le K+1$,
we inductively define Phase $i$ as follows.

\medbreak

{\bf Phase $\mathrm{i}$.} We continue the burning algorithm on $\eta$
with the restriction that no vertex of $D_{K-i+1}$ is allowed to burn.
That is, we set
\eqnsplst
{ B^{(i)}_0
  &:= \cup_{j \ge 0} B^{(i-1)}_j, \\
  U^{(i)}_0
  &:= \Lambda \setminus B^{(i)}_0, }
and for $j \ge 1$ we inductively set:
\eqnsplst
{ B^{(i)}_j
  &:= \left\{ v \in U^{(i)}_{j-1} \setminus D_{K-i+1} : \eta(v) \ge \deg_{U^{(i)}_{j-1}}(v) \right\}, \\
  U^{(i)}_j
  &:= U^{(i)}_{j-1} \setminus B^{(i)}_j. }
We have $B^{(i)}_j = \es$ eventually.
Note that if $i \le K$, there may be vertices in $\Lambda \setminus D_{K-i+1}$ that do not burn
in Phase $i$, only later.

\medbreak

Since $\eta$ is recurrent, all vertices that did not burn in Phases $1, \dots, K$, do burn
in Phase $K+1$ (if this was not true, we would have found a subset that is not ample for $\eta$). 
Hence we have $\cup_{j \ge 0} B^{(K+1)}_j = \Lambda \cup \{ s \}$ 

\medbreak

We now define a map $\varphi_{\cD,\Lambda} : \cR_\Lambda \to \cT_\Lambda$.
Regard $G_\Lambda$ as an oriented graph, with each edge being present
with both possible orientations. We fix for each $v \in \Lambda$ a linear ordering $\prec_v$ of the oriented edges $e$ such that $\tail(e) = v$. Given the burning of $\eta$ as above, we define
what oriented edges will be present in the tree $t = \varphi_{\cD,\Lambda}(\eta)$.

If $v \in B^{(i)}_j$ for some $1 \le i \le K+1$ and $j \ge 1$, then we place
an oriented edge pointing from $v$ to some $w \in B^{(i)}_{j-1}$. 
In the case $j = 1$ such edge exists, because $v$ must 
have a neighbour outside $U^{(i)}_0$, and hence in $B^{(i)}_0$.
In the case $j \ge 2$ such edge also exists, because the requirement to
burn $v$ at step $j$ implies that the degree of $v$ in $U^{(i)}_{j-1}$ is strictly
smaller than its degree in $U^{(i)}_{j-2}$. Hence $v$ has a neighbour 
in $B^{(i)}_{j-1} = U^{(i)}_{j-2} \setminus U^{(i)}_{j-1}$.
If there is more than one $w \in B^{(i)}_{j-1}$ neighbouring $v$, we make the 
choice of the edge dependent on $\eta(v)$, similarly to the usual burning bijection. 
Formally, we let:
\eqnsplst
{ m_v
  &:= \Big| \Big\{ f : \tail(f) = v,\, \head(f) \in \bigcup_{j' < j} B^{(i)}_{j'} \Big\} \Big|, \\
  F_v
  &:= \left\{ e : \tail(e) = v,\, \head(e) \in B^{(i)}_{j-1} \right\}. }
Due to the burning rule, we have
\eqnst
{ \eta(v)
  = \deg(v) - m_v + \ell \quad \text{ for some $0 \le \ell < |F_v|$.} }
With $\ell$ as above, let $e_v \in F_v$ be that edge $e$ such that $| \{ f \in F_v : f \prec_v e \} | = \ell$. Then we place the directed edge
$e_v$ in $t$.

\begin{lemma}
\label{lem:psi}
For any $\eta \in \cR_\Lambda$ the collection of edges $t$ (disregarding their orientations) is a spanning tree of $G_\Lambda$, and the map
$\varphi_{\cD,\Lambda} : \eta \mapsto t$ is injective. Consequently, $\varphi_{\cD,\Lambda}$ is a bijection between $\cR_\Lambda$ and $\cT_\Lambda$.
\end{lemma}

\begin{proof}
It is clear from the definitions that there are no cycles in $t$, since the sets $B^{(i)}_j$, are disjoint and ``lexicographically ordered''
by the indices $(i,j)$ for $1 \le i \le K+1$, $j \ge 1$. In order to show injectivity, suppose that $\eta_1 \not= \eta_2$. There is a first
time $(i,j)$ in the burning processes of $\eta_1$ and $\eta_2$, where the ``two processes differ''. That is, there exists a lexicographically
smallest $(i,j)$ such that $B^{(i')}_{j'} (\eta_1) = B^{(i')}_{j'} (\eta_2)$ for all $i' < i$, $j' \ge 1$
and for all $i' = i$, $j' < j$, and $\eta_1(v) = \eta_2(v)$ for all elements $v$ of these sets, but \eqnsplst
{ \text{there exists $v \in B^{(i)}_j (\eta_1) \cup B^{(i)}_j (\eta_2)$
    such that $\eta_1(v) \not= \eta_2(v)$.} }
It is easy to check that our definition of $\varphi_{\cD,\Lambda}$ assigns different oriented edges emanating from $v$ for $\eta_1$ and $\eta_2$. Since all edges are oriented towards the sink, this implies that the two trees also differ as unoriented trees, proving injectivity. Since $\cR_\Lambda$ and $\cT_\Lambda$ have the same number of elements
$\det(\Delta)$ \cite{Dhar(90)}, it follows that $\varphi_{\cD,\Lambda}$ is a bijection.
\end{proof}

Given $\eta \in \cR_\Lambda$, we define the \emph{burning time} $\tau^\cD(x,\Lambda;\eta)$ as the index 
of the pair $(i,j)$ in the lexicographic order, where $B^{(i)}_j \ni x$, $1 \le i \le K+1$, $j \ge 1$ 
(we restrict to the non-empty $B^{(i)}_j$'s). Note that in general this differs from the graph distance 
of $x$ from $s$ in the tree $\varphi_{\cD,\Lambda}(\eta)$.
This is because at Step 1 of Phase $i$, we may be connecting a vertex $v \in B^{(i)}_1 \cap D_{K-i+2}$ 
to a vertex $w$ that was burnt not in the last step of Phase $i-1$. 

Given $D \subset \Lambda$ and a spanning tree $t$ of $G_\Lambda$, we write $\desc_t(D)$ for the set of descendants of $D$ in $t$, that is, the collection of vertices $w$ such that the path in $t$ from $w$ to $s$ has a vertex in $D$.

\begin{lemma}
\label{lem:U=desc}
For any finite $\Lambda \subset V$, $1 \le i \le K+1$, and $\eta \in \cR_\Lambda$, the set of vertices that did not burn
by the end of Phase $i$ are precisely the descendants of $D_{K-i+1}$. That is, we have $U^{(i+1)}_0 = \desc_{\varphi_{\cD,\Lambda}(\eta)}(D_{K-i+1})$.
\end{lemma}

\begin{proof}
Observe that all vertices in $B^{(i+1)}_1$ are in $D_{K-i+1}$, otherwise
they could have been burnt in Phase $i$. Since the oriented edges assigned
by the bijection respect the lexicographic order, and the orientation is towards
the sink, this implies that all vertices burnt in Phases $i+1, \dots, K+1$ are in $\desc_{\varphi_{\cD,\Lambda}(\eta)}(D_{K-i+1})$. On the other hand, if a vertex $v$
was burnt in one of the Phases $1, 2, \dots, i$, then all vertices on the
oriented path from $v$ to $s$ were also burnt in one of these Phases, and hence
$v \not\in \desc_{\varphi_{\cD,\Lambda}(\eta)}(D_{K-i+1})$. This completes the proof.
\end{proof}

We next formulate a consistency property between the sandpile configurations
on the sets $\desc_{\varphi_{\cD,\Lambda}(\eta)}(D_{k})$, $k \ge 1$, that will help us to take the limit $\Lambda \uparrow V$.

\begin{definition}
\label{def:G*}
Given $k \ge 1$ and a finite simply connected set $W$ with $D_k \subset W \subset V$, we define the graph $G_{W,k}^* = (W \cup \{ s \}, \sfE_{W,k}^*)$ as follows. It contains all the edges
that $W$ induces in the graph $V$, and for each edge $e \in \sfE$ that connects
a vertex $u \in D_k$ with a vertex $v \in V \setminus W$, there is an edge
in $\sfE_{W,k}^*$ between $u$ and $s$. Note that there is a natural identification between $\sfE_{W,k}^*$ and a subset of $\sfE$, and we will use this identification
freely in what follows.
\end{definition}

\begin{lemma}
\label{lem:consistent}
(i) Suppose $D_k \subset W \subset V$ with $W$ simply connected.
There is a mapping $\psi_{W,k} : \cT_{G_{W,k}^*} \to \cR_{G_{W,k}^*}$
such that whenever $\Lambda \supset W$, $t \in \cT_\Lambda$ and
$W = \desc_t(D_k)$ holds,
\eqnst
{ \text{the restriction of the sandpile $\varphi^{-1}_{\cD,\Lambda}(t)$ to $W$
    equals $\psi_{W,k}(t_{W,k})$,} } where $t_{W,k}$ denotes the restriction of $t$ to the edges in $\sfE_{W,k}^*$.\\
(ii) Suppose $D_{k'} \subset D_{k} \subset W$. Let $t \in \cT_{G_{W,k}^*}$. If $W' = \desc_t(D_{k'})$, then
\eqnst
{ \text{the restriction of $\psi_{W,k}(t)$ to $W'$ is given by
     $\psi_{W',k'}(t_{W',k'})$.} }
\end{lemma}

\begin{proof}
(i) Write $\eta = \varphi^{-1}_{\cD,\Lambda}(t)$.
Due to Lemma \ref{lem:U=desc}, the statement $W = \desc_t (D_k)$ is equivalent
to the statement that in the sandpile $\eta$, $W$ is precisely
the set of vertices that did not burn in Phase $K - k + 1$. It is easy to check using the burning rules that as $\eta$ varies over all sandpiles with this
property, the restriction $\eta_W$ ranges over $\cR_{G_{W,k}^*}$, and $t_{W,k}$ is a spanning tree of $G_{W,k}^*$. It follows from our definition
of Phases $K - k + 2, \dots, K+1$ of the anchored bijection that $t_{W,k}$ is entirely determined by $\eta_W$, in a way independent of $\Lambda$.
The map $\eta_W \mapsto t_{W,k}$ is injective, and since $|\cR_{G_{W,k}^*}| = |\cT_{G_{W,k}^*}|$ it is bijective. Hence $\psi_{W,k}$ can be defined as the inverse of this map.

(ii) This follows similarly to part (i), because if $\Lambda \supset W$ and
$\eta$ is as in part (i), then the restriction of $\eta_W$ to $W'$ is $\eta_{W'}$.
\end{proof}

\medbreak

We are now ready to extend the bijection to $G$.

\emph{Anchored bijection on $G$.}

Observe that for every $t \in \cT$ and $v \in V$ there is a unique infinite path in $t$ starting at $v$. Hence for any finite $D \subset V$, we can define $\desc_t(D)$ as those vertices for which the infinite path starting at $v$ has a vertex in $D$.

Given $t \in \cT$, for every $k \ge 1$ let $W_k = \desc_t(D_k)$. Observe that
due to the one-end property \eqref{e:one-end} of elements of $\cT$, $W_k$ is finite
for all $k \ge 1$.
Denote by $t_{W_k,k}$ the restriction of $t$ to the edges in $\sfE_{W_k,k}^*$.
Due to Lemma \ref{lem:consistent}(ii), the configurations
$\psi_{W_k,k}(t_{W_k,k})$ consistently define a stable configuration $\eta$ on $V$. This $\eta$ will be an element of $\cR$, because for any finite $F \subset V$ there exists $k \ge 1$ such that $D_k \supset F$, and $\psi_{W_k,k}(t_{W_k,k}) = \eta_{W_k}$ is ample for $F$. We denote the configuration obtained by $\psi_\cD(t)$, so $\psi_\cD : \cT \to \cR$.

\begin{remark}
\label{rem:burning}
Whenever $\Lambda \supset W_k = \desc_t(D_k)$, we have the following property.
If we start burning $\psi_\cD(t)\vert_\Lambda$ with the restriction that no vertex of $D_k$ is allowed to burn, then the set of vertices that cannot be burnt is exactly $W_k$. This follows by considering the burning
process in some $W_{k'} \supset \Lambda$.
\end{remark}

\begin{lemma}
\label{lem:inj}
The map $\psi_\cD$ is injective and continuous.
\end{lemma}

\begin{proof}
Suppose that $t_1, t_2 \in \cT$ such that $\psi_\cD(t_1) = \psi_\cD(t_2)$. Let us denote $W^{(1)}_k = \desc_{t_1}(D_k)$ and $W^{(2)}_k = \desc_{t_2}(D_k)$,
and let $\Lambda = W^{(1)}_k \cup W^{(2)}_k$. By Remark \ref{rem:burning}, if we start the burning
process on $\psi_\cD(t_1) \vert_\Lambda = \psi_\cD(t_2) \vert_\Lambda$ in $\Lambda$ (with the
restriction that $D_k$ is not allowed to burn), then the set of vertices that do not burn
equals both $W^{(1)}_k$ and $W^{(2)}_k$. In particular, these sets are equal, that is, $W^{(1)}_k = W^{(2)}_k$. Denoting their common value by $W_k$, we have
\eqnst
{ \psi_{W_k,k}(t_1\vert_{\sfE^*_{W_k,k}})
  = \psi_\cD(t_1)\vert_{W_k}
  = \psi_\cD(t_2)\vert_{W_k}
  = \psi_{W_k,k}(t_2\vert_{\sfE^*_{W_k,k}}). }
Hence $t_1$ equals $t_2$ on $\sfE^*_{W_k,k}$. Since $k$ is arbitrary, it follows
that $t_1 = t_2$, and therefore $\psi_\cD$ is injective.

In order to see continuity, fix $t \in \cT$, let $\eta = \psi_\cD(t)$, and let
$k \ge 1$ be fixed. Let $W_k = \desc_t(D_k)$. Suppose that $t' \in \cT$ has the property that $t'$ agrees with $t$ on all edges in $\sfE$ that have an end vertex in $W_k$. Then it follows that $\desc_{t'}(D_k) = W_k$, and $t'_{W_k,k} = t_{W_k,k}$. Therefore \eqnst
{ \psi_\cD(t')\vert_{W_k}
  = \psi_{W_k,k}(t'_{W_k,k})
  = \psi_{W_k,k} (t_{W_k,k})
  = \psi_\cD(t)\vert_{W_k}. }
Since $k \ge 1$ is arbitrary, $W_k \supset D_k$ and $\cup_{k \ge 1} D_k = V$, this implies
continuity of $\psi_\cD$.
\end{proof}

The following lemma follows directly from the proof of \cite[Theorem 3]{JW(12)}. We provide a
sketch of the proof for the reader's convenience.

\begin{lemma}
\label{lem:conv}
The image of $\mathsf{WSF}$ under $\psi_\cD$ equals $\nu = \lim_{\Lambda \uparrow V} \nu_\Lambda$.
\end{lemma}
\begin{proof}[Sketch of the proof.]
Let $E$ be a cylinder event that only depends on the sandpile heights in $D_k$ for some $k \ge 1$. 
For any $\Lambda \supset D_k$, let $W_{\Lambda,k}$ be the random set of vertices that are 
unburnt just before the phase in which we first allow vertices in $D_k$ to burn, that is,
$U^{(K-k+2)}_0$. Due to Lemma \ref{lem:U=desc}, $W_{\Lambda,k}$ also equals the set of 
descendants of $D_k$ in $\psi_{\cD,\Lambda}^{-1}(\eta_\Lambda)$, where $\eta_\Lambda$ is 
the sandpile configuration in $\Lambda$.
Recall the auxiulliary graph $G^*_{W,k}$ from Definition \ref{def:G*}.
Due to the proof of Lemma \ref{lem:consistent}(i), for any fixed set $D_k \subset W \subset \Lambda$,
the conditional distribution of $\eta_W$, given the event $\{ W_{\Lambda,k} = W \}$ is
given by $\nu_{G^*_{W,k}}$. Hence, conditioning on the value of $W_{\Lambda,k}$, we have:
\eqnspl{e:W-decomp}
{ \nu_\Lambda ( E )
  = \sum_{D_k \subset W \subset \Lambda} \nu_\Lambda ( W_{\Lambda,k} = W ) 
    \nu_{G^*_{W,k}} ( \eta_W \in E ). } 
Note that, in the notation of Lemma \ref{lem:consistent}, we have
\eqnsplst
{ \nu_{G^*_{W,k}}( \eta_W \in E )
  &= \UST_{G^*_{W,k}} ( t : \psi_{W,k}(t) \in E ) \\
  &= \WSF ( t : \psi_{W,k}(t_{W,k}) \in E \,|\, \desc_t(D_k) = W ) \\
  &= \WSF ( t : \psi_\cD(t) \in E \,|\, \desc_t(D_k) = W ). }
In particular, this probability does not depend on $\Lambda$.
We also have 
\eqnst
{ \lim_{\Lambda \uparrow V} \nu_\Lambda ( W_{\Lambda,k} = W ) 
  = \lim_{\Lambda \uparrow V} \UST_\Lambda ( t : \desc_t(D_k) = W )
  = \WSF ( t : \desc_t(D_k) = W ). }
This is because for a fixed finite set $W$, the event $\desc_t(D_k) = W$ is spanning-tree-local:
it only depends on the status of the edges in $\sfE^*_{W,k}$. Finally, note that due to the 
one-end property \eqref{e:one-end} the family $\{ W_{\Lambda,k} : \Lambda \supset D_k \}$ is tight, 
in the sense that 
\eqnst
{ \lim_{M \to \infty} \sup_{\Lambda \supset D_M} 
    \UST_\Lambda ( t : \desc_t(D_k) \not\subset D_M )
  = 0. }
This allows us to pass to the limit in \eqref{e:W-decomp} and obtain
\eqnsplst
{ \lim_{\Lambda \uparrow V} \nu_\Lambda(E)
  &= \nu(E) \\
  &= \sum_{\substack{W : \text{$W$ is finite} \\ W \supset D_k}} \WSF ( t : \desc_t(D_k) = W ) \, 
    \WSF ( \psi_\cD(t) \in E \,|\, \desc_t(D_k) = W ) \\
  &= \WSF ( t : \psi_\cD(t) \in E ). }
\end{proof}

Lemmas \ref{lem:inj}, \ref{lem:conv} imply Theorem \ref{thm:anchored-bijection}.

Our final lemma shows the coupling property \eqref{e:burn-couple}.

\begin{lemma}
\label{lem:couple}
Fix $o \in D_1$. For any $t \in \cT$ and $x \in V$ the limit
\eqnst
{ \lim_{\Lambda \uparrow V} \left[ \tau^\cD(x,\Lambda;\psi_\cD(t))
    - \tau^\cD(o,\Lambda;\psi_\cD(t)) \right]
  =: \tau^\cD(x;\eta) \in \Z }
exist.
\end{lemma}

\begin{proof}
Let $k \ge 1$ be the smallest index such that $x \in D_k$, let $W = \desc_t(D_k)$, and suppose that $\Lambda \supset W$.
Due to Remark \ref{rem:burning}, for any such $\Lambda$
the last $k+1$ phases of the burning of $\eta_\Lambda$
have identical history. This implies the claim.
\end{proof}

\section{Rate of convergence in $\mathbb{Z}^d$, $d\geq 3$.}
\label{sec:d>=3}

Henceforth we consider the graphs $G = \Z^d$, and in this section we assume $d \ge 3$.
Let $D_k$ be the intersection of the Euclidean ball of radius $k$ about the origin with $\Z^d$. We write $\partial W$ for the set of vertices in $W^c := \Z^d \setminus W$ that have a neighbour in $W$.

Let $\Lambda \subset \Z^d$ be finite. We consider the realizations of $\WSF$ and 
$\UST_{G_\Lambda}$ via stacks of arrows, as introduced by Wilson \cite{W(96)}; 
see also \cite{LPbook}. To each vertex $v \in \Z^d$ we associate an i.i.d.~sequence of arrows
$\{ e^v_i : i = 1, 2, \dots \}$, where $e^v_i$ is an oriented edge with $\tail(e^v_i) = v$ and $\head(e^v_i)$ uniformly 
random among the neighbours of $v$. The stacks associated to different $v$ are independent. 
We define $\P$ as the underlying probability measure for the stacks of arrows. We say that 
$e^v_i$ has colour $i$, and we envision $e^v_1$ lying directly above $e^v_2$ in the stack, 
and similarly, for all $k$, $e^v_k$ lying directly above $e^v_{k+1}$.
An oriented cycle $\cC$ in $\Z^d$ is associated the \emph{weight} $w(\cC) = (2d)^{-|\cC|}$,
where $|\cC|$ denotes the number of arrows in $\cC$. Sometimes
we will need to consider \emph{coloured cycles}, that is, a cycle consisting 
of some arrows $e^{v_1}_{i_1}, \dots, e^{v_r}_{i_r}$.
We will use bold characters, like $\bC$, to denote coloured cycles.
In this case, $\cC$ will denote the cycle obtained from $\bC$ by ignoring the colours.

Wilson's algorithm \cite{W(96)} is based on the idea of \emph{cycle popping} that we now decribe. 
We start with a configuration of stacks of arrows, as described above. We say that 
initially $e^v_i$ is \emph{in position $i$}. We refer to the arrows in position $1$ as
lying on top of the stack. Suppose that arrows $e^{v_1}_1, \dots, e^{v_r}_1$ on top of the
stacks form a coloured cycle $\bC$. By \emph{popping} $\bC$, we mean removing the arrows in $\bC$
from the stacks, and shifting the positions of the arrows beneath them upwards. That is:
after popping $\bC$, $e^{v_j}_k$ will be in position $k-1$ for $j = 1, \dots, r$,
$k \ge 2$. Similarly, if at any later time some arrows $e^{v_1}_{i_1}, \dots, e^{v_r}_{i_r}$ 
are all in position $1$ and form an oriented cycle $\bC$, we may pop them and shift
the arrows beneath them upwards.
 
As shown in \cite{W(96)}, with probability $1$, only finitely many coloured cycles 
contained in $\Lambda$ can be popped, and on this event, regardless of what order of 
popping is chosen, the same set of coloured cycles get popped. Moreover, the arrows 
that are left on top of the stacks when no more cycles can be popped form a spanning 
tree of $G_\Lambda$ (oriented towards $s$), that also does not depend on the order
of popping. Furthermore, the set of coloured cycles popped and the tree obtained
are independent, and the tree is distributed according to $\UST_{G_\Lambda}$.

Cycle popping can also be made sense of in $\Z^d$, $d \ge 3$. 
One way is to use loop-erased random walks (LERW), as in \cite[Theorem 5.1]{BLPS(01)},
known as Wilson's method rooted at infinity.
Given a finite path $\pi = [x_0, \dots, x_M]$ in $\Z^d$, its \emph{loop-erasure} $\LE(\pi)$
is defined by chronologically erasing cycles from the path, as they are 
created; see \cite{LPbook}. Loop-erasure also makes sense for infinite paths
$\pi$, as long as $\pi$ visits every vertex finitely often.
To describe Wilson's method rooted at infinity, order the vertices of $\Z^d$ 
arbitrarily as $v_1, v_2, \dots$. Starting from $v_1$, follow the arrows on top 
of the stacks, and whenever a cycle is completed, pop that cycle. The trajectory 
traced by this walk is a simple random walk $\{ S^{(1)}(m) \}_{m \ge 0}$ under $\P$, 
so due to transience, every vertex is visited only finitely many times, with probability $1$. 
Hence, on this event, there is a well-defined configuration of stacks of un-popped 
arrows, after the entire trajectory of $S^{(1)}$ has been traced. 
On top of the stacks now lie $\mathfrak{F}_1 := \LE(S^{(1)}[0,\infty))$, and unexamined
arrows everywhere else. Next, starting from $v_2$, again follow the arrows on top of the 
stacks, popping any cycles that are completed. The trace of the path will now be
a random walk $S^{(2)}[0,\tau^{(2)}]$, where $\tau^{(2)} \in [0,\infty]$ is the first 
hitting time of $\mathfrak{F}_1$. Upon hitting $\mathfrak{F}_1$, a segment of 
$\mathfrak{F}_1$ is retraced without encountering any further cycle, and on top of
the stacks will lie $\mathfrak{F}_2 := \mathfrak{F}_1 \cup \LE(S^{(2)}[0,\tau^{(2)}])$,
with unexamined arrows everywhere else. Continue this way with $v_3, v_4, \dots$.
With probabilty one, from each stack only finitely many arrows get popped, 
hence the procedure reveals a random spanning forest $T$. Due to \cite[Theorem 5.1]{BLPS(01)},
$T$ is distributed according to $\WSF$. We will also need the following alternative way of
popping cycles in $\Z^d$:
\eqn{e:popping}
{ \parbox{13cm}{first pop all cycles contained in $D_1$, then pop all cycles 
    contained in $D_2$, etc.} } 
Wilson's proof for finite graphs \cite{W(96)} can be adapted to show that on the 
probability $1$ event when $T$ is well-defined, the procedure \eqref{e:popping}
reveals exactly the same forest $T$. In particular, for any finite $\Lambda \subset \Z^d$, 
cycle popping in $\Lambda$ also terminates with probability $1$, resulting in
a spanning tree $T_\Lambda$, distributed according to $\UST_{G_\Lambda}$. 
Thus, using the same stacks of arrows for cycle popping 
in $\Lambda$ and in $\Z^d$ provides the required coupling of $\WSF$ 
and $\UST_{G_{\Lambda}}$.

Given a cylinder event $E \subset \{0, \dots, 2d-1 \}^{D_k}$ only depending
on sandpile heights in $D_k$, let us write $E_{\Z^d} = \{ \psi_\cD(T) \in E \}$ 
and $E_\Lambda = \{ \psi_{\cD,\Lambda}(T_\Lambda) \in E \}$.
We have $\P (E_\Lambda) = \nu_\Lambda(E)$, due to Lemma \ref{lem:psi} and 
$\P (E_{\Z^d}) = \nu(E)$, due to Lemma \ref{lem:conv}.

\begin{theorem}
\label{thrm:d>2}
Let $E$ be a cyclinder event depending only on the sandpile heights in $D_k$.
Let $d \ge 3$, let $\Lambda \subset \Z^d$ be a finite set and let $N$ be the radius of the largest ball centered at the origin that is contained in $\Lambda$. We have
\eqnst
{ \left| \nu_\Lambda(E) - \nu(E) \right|
  \le \P( E_\Lambda \Delta E_{\Z^d} )
  \le \begin{cases}
      C_d k^{d-1} N^\frac{2-d}{2d} & \text{if $d \geq 5$;} \\
      C k^{26/9} N^{-2/9}          & \text{if $d = 4$;} \\
      C k^{17/9} N^{-1/9}          & \text{if $d = 3$.}
      \end{cases} }
Here $\Delta$ denotes symmetric difference.
\end{theorem}

The proof is broken down into a number of propositions and lemmas.
Let us write $W_k$ for the random set of descendants of $D_k$ in $T$,

\begin{proposition}
\label{prop:W_k-bnd}
Suppose $d \ge 3$, $1 \le k < n < N$, and $\Lambda \supset D_{N}$.
There is a constant $C_d > 0$ such that
\eqn{e:WSF-bnd}
{ \P \left( \text{$W_k \subset D_n$ but $W_k \not= W_{k,\Lambda}$ or
     $T\vert_{\sfE^*_{W,k}} \not= T_\Lambda\vert_{\sfE^*_{W,k}}$} \right)
  \le C_d \frac{k^{d-2} n^2}{(N-n)^{d-2}}. }
\end{proposition}

\begin{proof}
If we successively pop all cycles in $D_n$, then in $D_{n+1}$, then in $D_{n+2}$, etc.,
then we see that $\P$-a.s.~on the event $W_k \subset D_n$ we have $W_{k,\Lambda'} = W_k$
and $T\vert_{E^*_{W,k}} = T_\Lambda\vert_{E^*_{W,k}}$ for all large enough finite $\Lambda'$. Therefore, it is enough to show that for all finite $\Lambda' \supset \Lambda$ we have
\eqn{e:WSF-bnd2}
{ \P \left( \text{$W_{k,\Lambda'} \subset D_n$ but $W_{k,\Lambda'} \not= W_{k,\Lambda}$
     or $T_{\Lambda'}\vert_{\sfE^*_{W,k}} \not= T_\Lambda\vert_{\sfE^*_{W,k}}$} \right)
  \le C_d \frac{k^{d-2} n^2}{N^{d-2}}, }
with $C_d$ independent of $\Lambda$, $\Lambda'$.

In order to prove \eqref{e:WSF-bnd2}, we first pop all cycles we can that are contained 
in $\Lambda$. This leaves on top of the stacks in $\Lambda$ the wired spanning tree 
$T_\Lambda$ of $G_\Lambda$. Let $\fL$ denote the collection of remaining coloured cycles 
contained in $\Lambda'$ that need to be popped in order to obtain the wired spanning tree 
$T_{\Lambda'}$ in $\Lambda'$. For convenience, the cycles in $\fL$ are regarded as having 
colours according to their \emph{current positions} in the stacks, i.e.~after all cycles 
contained in $\Lambda$ have been popped. We claim that the probability distribution of 
$\fL$ is proportional to total weight and that $\fL$ is independent 
of the wired spanning tree $T_{\Lambda'}$ in $\Lambda'$, that is:
\eqn{e:factorize}
{ \P ( \fL = \{ \bC_1, \dots, \bC_K \},\, T_{\Lambda'} = t_{\Lambda'} )
  = \UST_{G_{\Lambda'}}(t_{\Lambda'}) \frac{1}{Z} \prod_{j=1}^K w(\cC_j), }
where $Z$ is a normalization factor. Indeed, we show that this follows 
from Wilson's theorem \cite{W(96)}. Let us write $\fL^0_\Lambda$, respectively
$\fL^0_{\Lambda'}$, for the collection of coloured cycles contained in 
$\Lambda$, respectively $\Lambda'$, that we need to pop in order
to reveal $T_\Lambda$, respectively $T_{\Lambda'}$. Then $\fL$ is a deterministic
function of $\fL^0_{\Lambda'}$ (recall that the colours of cycles in $\fL$
are according to their positions aquired after cycle popping in $\Lambda$
is complete). By Wilson's theorem, $T_{\Lambda'}$ is independent of $\fL^0_{\Lambda'}$,
and hence of $\fL$, and is distributed according to $\UST_{G_{\Lambda'}}$. 
Therefore, the left hand side of \eqref{e:factorize} equals
\eqnst
{ \UST_{G_{\Lambda'}}(t_{\Lambda'}) \, \P ( \fL = \{ \bC_1, \dots, \bC_K \} ). }
In order to show that the second factor is proportional to weight, first
observe that $\fL^0_{\Lambda'}$ and the pair $(\fL^0_\Lambda, \fL)$ are deterministic
functions of each other. We show that $\fL^0_\Lambda$ and $\fL$ are independent. 
This is because, using Wilson's Theorem again, $\fL^0_\Lambda$, $T_\Lambda$, the stacks 
of arrows beneath $T_\Lambda$, and the stacks of arrows in $\Lambda' \setminus \Lambda$ 
are mutually independent, and $\fL$ is a deterministic function of the latter three.
We have
\eqnst
{ \P ( \fL^0_\Lambda = \{ \bC^0_1, \dots, \bC^0_{K^0} \},\, \fL = \{ \bC_1, \dots, \bC_K \} )
  = \frac{1}{Z^0_{\Lambda'}} \times \prod_{\ell=1}^{K^0} w(\cC^0_\ell) \times \prod_{j=1}^K w(\cC_j). }
Summing over all instances of $\fL^0_\Lambda$, the independence of $\fL^0_\Lambda$ and 
$\fL$ implies 
\eqnst
{ \P ( \fL = \{ \bC_1, \dots, \bC_K \} ) 
  = \frac{1}{Z} \prod_{j=1}^K w(\cC_j). }
This proves the claim made in \eqref{e:factorize}


We introduce a partial order on elements of $\fL$ as follows:
we say that $\bC \prec \bC'$, if there exist $j \ge 1$ and a sequence of coloured cycles 
$\bC = \bC_j, \bC_{j-1}, \dots, \bC_0 = \bC'$ all in $\fL$, such that for each $1 \le r \le j$, 
the coloured cycles $\bC_{r-1}$ and $\bC_r$ share at least one vertex whose colour in 
$\bC_r$ is one greater than its colour in $\bC_{r-1}$. The meaning of the relation
$\prec$ is the following: 
\eqn{e:pop-equiv}
{ \text{$\bC \prec \bC'$} \qquad \Longleftrightarrow \qquad
  \text{regardless of the order of popping, $\bC'$ is popped before $\bC$.} } 
(Recall that the set $\fL$ does not depend on the order of popping.) The direction 
$\Longrightarrow$ of this equivalence is immediate from the definition of $\prec$.
To see the $\Longleftarrow$ direction, let us pop every cycle we can without popping
$\bC'$. This does not reveal $\bC$. Now pop $\bC'$, and note that any cycle that is
revealed as a result of popping $\bC'$ necessarily shares a vertex with $\bC'$. Popping 
further cycles it holds that any cycle that is revealed has a chain of cycles 
leading to $\bC'$. In particular, $\bC$ must have this property.
The equivalence \eqref{e:pop-equiv} makes it clear that $\prec$ is a partial order
on $\fL$.


We apply a parallel popping procedure to reveal $\fL$, defined in stages. In each stage, we pop all cycles on top of the stacks, simultaneously.
If the event on the left hand side of \eqref{e:WSF-bnd2} occurs,
there exists a smallest integer $\ell \ge 1$, such that in stage
$\ell$ we pop some cycle that intersects $\overline{W} := W_{k,\Lambda'} \cup \partial W_{k,\Lambda'}$.
Indeed, if we never popped any such cycles, then the arrows
attached to all the vertices in $\overline{W}$ would have the same
direction as they had in $T_\Lambda$, which would force $W_{k,\Lambda'} = W_{k,\Lambda}$
and $T_{\Lambda'}\vert_{\sfE^*_{W,k}} = T_\Lambda\vert_{\sfE^*_{W,k}}$.
Let us select, according to some fixed arbitrary rule, a cycle
$\bD_1 \in \fL$ popped in stage $\ell$, and a vertex $w \in \bD_1 \cap \overline{W}$. Let \eqnspl{e:M-def}
{ \fM
  &:= \{ \bD \in \fL : \bD \succeq \bD_1 \}. }
Observe that $\fM$ can be popped from $\fL$ (without popping
any other cycles), since by construction, $\fM$ is closed
under domination in the partial order $\prec$. Define $\widetilde{\fL}$ 
to be the collection of coloured cycles left after popping $\fM$ from $\fL$.

\begin{lemma}
\label{lem:fM-inj}
The map $\fL \mapsto (\fM, \widetilde{\fL})$ is injective. \qed
\end{lemma}

The proof is omitted as it immediately follows from the definition of the map. 

We are going to join the cycles in $\fM$ into a single loop $\gamma$ in $\Z^d$, and then bound the probability of the
possible arising loops in Lemma \ref{lem:super-loop-bnd}
below. We set $\gamma(0) = w$. Note that by the definition of
$\bD_1$, the arrow at $w$ is at the top of its stack.
We define $\gamma$ by following the arrows, starting with the one on the top of the stack of $w$, and whenever we visit a vertex $v$ for the $i$-th time, we use the $i$-th coloured arrow at $v$. The walk stops upon the first return to $w$. We call $\gamma$ the \emph{loop associated to $\fM$}.
The purpose of the next lemma is to show that $\gamma$ is well-defined
and the map $\fM \to \gamma$ is injective.

\begin{lemma}
\label{lem:walk} Let $W \subset D_n$ be a fixed set and let $w \in \partial W$
be a fixed vertex. Suppose that $\fL$ is a collection of coloured cycles that can be popped, and $\bD_1 \in \fL$ has the property that $w \in \cD_1$,
but no coloured cycle popped at any earlier stage than $\bD_1$ intersects $\overline{W} = W \cup \partial W$. Let $\fM$ be defined by formula \eqref{e:M-def}. Then we have: \\
(i) The loop associated to $\fM$ is well-defined in that the walk does return to $w$.\\
(ii) Every coloured edge in $\fM$ is used exactly once by the loop. \\
(iii) The map $\fM \mapsto \gamma$ is injective.
\end{lemma}

\begin{proof}
(i), (ii) We prove the two statements together by induction on the number of cycles in $\fM$.
If $\fM$ consists of the single cycle $\bD_1$, the statement is trivial.
Otherwise, consider the first time we return to a vertex $v$
that we visited before. Then the cycle just found, $\bD$, say, is necessarily on top of the stacks and $\bD \not= \bD_1$. Also, since the walk starts with an
arrow belonging to a cycle in $\fM$, it is easy to see that $\bD \in \fM$. Now pop $\bD$, and define $\fL'$, $\fM'$, $\bD_1'$ by moving
the arrows in the stacks of the vertices of $\cD$ up by one (and removing
the arrows in $\bD_1$). Observe that $\fL'$, $\fM'$, $\bD_1'$ also satisfy
the hypotheses of the Lemma, so by the induction hypothesis, the walk $\gamma'$ defined by $\fM'$ visits each arrow of $\fM'$ exactly once. Hence inserting into $\gamma'$ the cycle $\cD$ at $v$
we get the walk $\gamma$ defined by $\fM$. This implies the statements (i) and (ii).

(iii) This follows from the fact that by construction, following the
history of the loop-erasure process on $\gamma$ (started at $w$)
the loops erased are precisely the loops in $\fM$.
\end{proof}

We continue with the proof of Proposition \ref{prop:W_k-bnd}. We bound
the left hand side of \eqref{e:WSF-bnd2} from above as follows.
Let $\Pi$ denote the class of all sets of coloured loops that are
possible values of $\fL$. Let $\Gamma_w$ denote the collection of loops in $\Z^d$ that start and end at $w$ and visit $\Lambda^c$.
Let $\Gamma_{w,\Lambda'}$ denote those loops in $\Gamma_w$ that stay inside $\Lambda'$. By the stated independence of the spanning tree in $\Lambda'$ and $\fL$, we have
\eqnspl{e:decompose}
{ &\P \left( \text{$W_{k,\Lambda'} \subset D_n$ and $W_{k,\Lambda'} \not= W_{k,\Lambda}$ or
    $T_{\Lambda'}\vert_{\sfE^*_{W,k}} \not= T_\Lambda\vert_{\sfE^*_{W,k}}$} \right) \\
  &\qquad \le \sum_{W \subset D_n} \mu_{\Lambda'}(W_{k,\Lambda'} = W)
    \frac{1}{Z} \ \sum_{\substack{\fL \in \Pi : \exists \bD_1 \in \fL,\\ \bD_1 \cap \partial W \not= \es}}
    \ \prod_{\bC \in \fL} w(\cC). }
We fix $W$, and estimate the sum over $\fL$. To every $\fL$ occurring in the
sum, we have associated (by our arbitrary rule), a choice of $w \in \partial W$
and $\fM \subset \fL$ containing $w$. This $\fM$, in turn determines a loop $\gamma$ based at $w$. Observe that \eqnst
{ \prod_{\bC \in \fL} w(\cC)
  = \prod_{\bD \in \fM} w(\cD) \times
    \prod_{\widetilde{\bC} \in \widetilde{\fL}} w(\widetilde{\cC})
  = w(\gamma) \times \prod_{\widetilde{\bC} \in \widetilde{\fL}} w(\widetilde{\cC}). }
Hence, using the injectivity statements in Lemma \ref{lem:fM-inj} and Lemma \ref{lem:walk}(iii), the right hand side of \eqref{e:decompose} is at most
\eqnspl{e:bnd-with-loops}
{ \frac{1}{Z} \sum_{w \in \partial W} \sum_{\gamma \in \Gamma_{w,\Lambda'}}
    w(\gamma) \sum_{\widetilde{\fL} \in \Pi}
    \prod_{\widetilde{\bC} \in \widetilde{\fL}} w(\widetilde{\cC})
  &\le \sum_{w \in \partial W} \sum_{\gamma \in \Gamma_{w,\Lambda'}}
    w(\gamma) \\
  &\le \sum_{w \in \partial W} \sum_{\gamma \in \Gamma_w}
    w(\gamma). }

\begin{lemma}
\label{lem:super-loop-bnd}
For any $w \in D_n$, we have
\eqn{e:super-loop-bnd}
{ \sum_{\gamma \in \Gamma_w}
    w(\gamma)
  \le \frac{C_d}{(N-n)^{d-2}}. } \end{lemma}

\begin{proof}
The weight of a loop is equal to the probability of each step present occurring. Therefore the sum of the weights over loops $\Gamma_w$ equals the sum of the probabilities of random walk paths that start and end at $w$ and exit $\Lambda$. Letting $S$ denote simple random walk and
$\tau_N$ the first exit time of $D_N$ we get
\begin{align*}
  \sum_{\gamma \in \Gamma_w}  w(\gamma)
  &= \sum_{m \ge 0} \sum_{z \in \partial D_N} \sum_{r > m}
  \P^w (\tau_N = m,\, S(m) = z) \P^w ( S(r) = w,\,|\, \tau_N = m,\, S(m)=z )\\
  &= \sum_{m \ge 0} \sum_{z \in \partial D_N} \P^w( \tau_N = m,\, S(m)=z ) G(z,w)\\
  &\leq \frac{C_d}{(N-n)^{d-2}} \sum_{m \ge 0} \sum_{z \in \partial D_N}
     \P^w( \tau_N = m,\, S(m)=z )\\
  &= \frac{C_d}{(N-n)^{d-2}}.
\end{align*}
Here $G(z,w)$ is Green's function, see \cite[Section 4.3]{LLbook} for a proof of the bound on $G(z,w)$. \end{proof}

Inserting \eqref{e:super-loop-bnd} and \eqref{e:bnd-with-loops} into
\eqref{e:decompose} we get
\eqnspl{e:W-bnd}
{ &\P \left( \text{$W_{k,\Lambda'} \subset D_n$ and $W_{k,\Lambda'} \not= W_{k,\Lambda}$
     or $T_{\Lambda'}\vert_{\sfE^*_{W,k}} \not= T_\Lambda\vert_{\sfE^*_{W,k}}$} \right) \\
  &\qquad \le \frac{C_d}{(N-n)^{d-2}} \E_{\mu_{\Lambda'}} \left[ |\partial W_{k,\Lambda'}| : W_{k,\Lambda'} \subset D_n \right] \\
  &\qquad \le \frac{C_d}{(N-n)^{d-2}} \E_{\mu_{\Lambda'}} \left[ |W_{k,\Lambda'}| : W_{k,\Lambda'} \subset D_n \right]. }
We estimate the right hand side in the last equation in the following lemma.

\begin{lemma}
\label{lem:size-of-descendants}
We have
\eqnst
{ \E_{\mu_{\Lambda'}} \left[ |W_{k,\Lambda'}| : W_{k,\Lambda'} \subset D_n \right]
  \le C_d k^{d-2} n^2. }
\end{lemma}

\begin{proof}
By Wilson's algorithm, the probability that a vertex $x \in D_n \setminus D_k$
is in $W_{k,\Lambda'}$ is at most the probability that simple random walk started at $x$ hits $D_k$. This is bounded by $C_d k^{d-2}/|x|^{d-2}$.
Summing over $x \in D_n$ gives
\eqnsplst
{ \E_{\mu_{\Lambda'}} \left[ |W_{k,\Lambda'}| : W_{k,\Lambda'} \subset D_n \right]
  &\le | D_k | + \E_{\mu_{\Lambda'}} \left[ |W_{k,\Lambda'} \cap (D_n \setminus D_k)|\right] \\
  &\le C_d k^d + C_d n^2 k^{d-2} \\
  &\le C_d n^2 k^{d-2}. }
\end{proof}

The above lemma and \eqref{e:W-bnd} completes the proof of Proposition \ref{prop:W_k-bnd}. \end{proof}

\begin{proposition}
\label{prop:cond}
Suppose $d\ge 3$. Then for sufficiently large $n$ we have
\begin{equation*}
  \P( W_k \not \subset D_n )
  \leq C_d k^{d-1} n^{\frac{2-d}{2d}}.
\end{equation*}
\end{proposition}

We prove this proposition by extending the argument of \cite[Theorem 4.1]{LMS(08)},
that requires a couple of alterations.

\begin{proof}
Condition on the event that the restriction of the uniform spanning forest to $D_k$,
denoted $T\vert_{D_k}$, is a fixed forest $K$. Let $K_j$, $j = 1, 2, \dots$ denote the connected components of $K$. Then
\begin{align*}
  \P( \desc(D_k) \not\subset D_n \,|\, T\vert_{D_k} = K )
  &= \P \big( \cup_j \big\{ \desc(K_j) \not\subset D_n \big\} \,\big|\, T\vert_{D_k} = K \big) \\
  &\leq \sum_j \P( \desc(K_j) \not\subset D_n \,|\, T\vert_{D_k} = K).
\end{align*}
In order to deal with the summand in the last expression, we need to generalize \cite[Lemma 3.2]{LMS(08)}.
Given a graph $G$, and $V$ a subset of the vertices, we denote by $G/V$ the graph obtained from $G$ by
identifying all the vertices in $V$ to a single vertex and removing loop-edges.

\begin{lemma}
\label{lem:domination}
Let $G$ be a finite graph containing $D_k$ as a subgraph and $s$ a vertex of $G$ with $s \not\in D_k$.
Let $T_K$ denote the uniform spanning tree of $G$ conditioned on its restriction to $D_k$ being $K$.
Let $L_j(T_K)$ denote the unique path from $K_j$ to $s$ in $T_K$. Then on the set of edges
not belonging to $K_j$, the graph $T_K \setminus L_j(T_K)$ is stochastically dominated by
the uniform spanning tree of $G / (K_j \cup \{s\})$, conditioned on the event that its restriction to $D_k/K_j$ equals $K/K_j$.
\end{lemma}

\begin{proof}
First we further condition on $L_j(T_K) = L$. Note that under this conditioning, $T_K \setminus L$ has the same distribution as the uniform spanning tree of
$G/\Vrt(L)$ given $K$, where $\Vrt(\cdot)$ denotes vertex set of a graph. By the negative association theorem of Feder and Mihail \cite{FM(92)}, \cite[Chapter 4]{LPbook}, conditioning on an edge being present makes the remaining set of edges stochastically smaller. As $\Vrt(L)$ contains both $K_j$ and $s$ we can repeatedly apply this result to deduce that on the edges not belonging to $K_j \cup L$ the set of edges $T_K \setminus L$
is dominated by the uniform spanning tree of $G / (K_j \cup \{ s \})$ given $K / K_j$.
We can now average over all possible paths $L = L_j(T_K)$ to remove this part of the conditioning and get the stated lemma. \end{proof}

We will use the following corollary of Lemma \ref{lem:domination} that can be deduced by taking
weak limits. Let $\mathfrak{F}_{K,j}$ denote the wired spanning forest conditioned on $K$ with
$K_j$ wired to infinity (defined as the weak limit of uniform spanning trees conditioned on $K$ with
$K_j$ wired to the sink).

\eqnsplst
{ \parbox{13cm}{The set of descendants of $K_j$ in the wired uniform spanning forest
  conditioned on $K$ is stochastically dominated by the connected component of $K_j$
  in $\mathfrak{F}_{K,j}$.} }

The rest of the proof follows an outline similar to the proof of \cite[Theorem 4.1]{LMS(08)}.
We define edge sets $E_1 \subset E_2 \subset \dots$ as follows. Let $E_0 = K_j$. Assuming
$E_n$ has been defined, let $S_n$ be the set of vertices of the connected component of
$\mathfrak{F}_{K,j} \cap E_n$ containing $K_j$. If all edges incident with $S_n$ are in $E_n$, we set $E_{n+1} = E_n$. If not, let $e$ be an edge incident with $S_n$ that minimizes
$\min \{ r : e \subset B_r \}$, where $B_r = \{ x \in \Z^d : \| x \|_\infty \le r \}$,
and set \eqnst
{ E_{n+1}
  := \begin{cases}
     E_n \cup \{ e \} & \text{if $e$ does not connect $S_n$
         with a component $K_i$, $i \not= j$;} \\
     E_n \cup \{ e \} \cup K_i & \text{if $e$ connects $S_n$ with $K_i$.}
     \end{cases} }
When in the above $E_n \subset B_{r-1}$, i.e.~a ``new shell is visited'' by the process,
we make the further requirement that $e$ be the edge along which the unit
current flow from $S_n$ to $\infty$ is maximal.

Let $M_n$ be the effective conductance from $S_n$ to $\infty$ in the complement of
$E_n$, with the edges of $K$ shorted:
\eqnsplst
{ M_n
  := \mathcal{C}( \text{$S_n \leftrightarrow \infty$ in $(\mathbb{Z}^d/K) \setminus E_n$} ). }
Then by \cite[Lemma 3.3]{LMS(08)}, \cite[Theorem 7]{Mor(03)}, $(M_n)_{n \ge 0}$ is a martingale
with respect to the filtration $\mathcal{F}_n$ generated by $E_n$, $\mathfrak{F}_{j,K} \cap E_n$.

The $M_0$ term is no longer constant, as in the original proof. Nevertheless, the argument of \cite[Theorem 4.1]{LMS(08)} gives: \begin{equation*}
  \P ( \desc(K_j) \not\subset D_n \,|\, T\vert_{D_k} = K)
  \leq C_d n^\frac{2-d}{2d} M_0(K_j)
\end{equation*}
We now bound $M_0(K_j)$ still with the conditioning that on $D_k$ we have the forest $K$. Therefore we can work on the graph produced by deleting any edges from $D_k$ that do not appear in $K$ and contracting each component of $K$ to a distinct vertex. 
By definition, the effective conductance from $K_j$ to $\infty$ is the infimum of the energy of functions that are zero on $K_j$ and one except on finitely many vertices.
Therefore consider the function defined by $g(v)=0$ if $v \in K_j$ and one otherwise. This is clearly a valid function with regards to the infimum and will have energy equal to the number of edges connected to $K_j$. As all edges in $D_k$ that are not present in $K$ have been deleted and $K_j$ is a connected component of $K$, the only edges will be those connected to $K_j$ from the outside of $D_k$. The size of this set is at most $C_d |\partial D_k \cap K_j|$.

Summing over the connected components, and using the fact that the $K_j$'s are disjoint and cover all of $D_k$, we get \begin{equation*}
  \sum_j M_0(K_j)
  \le C_d |\partial D_k|
  \le C_d k^{d-1}.
\end{equation*}
Then as this bound is independent of $K$ we can average over all possible $K$ to get the unconditioned result:
\begin{align*}
 \P( \desc(D_k) \not\subset D_n )
 \leq C_d n^\frac{2-d}{2d}k^{d-1}.
\end{align*}
This completes the proof of Proposition \ref{prop:cond}.
\end{proof}

\begin{proof}[Proof of Theorem \ref{thrm:d>2}]
If $W_k = W_{k,\Lambda}$ and $T$ and $T_\Lambda$ agree on $E^*_{W,k}$, then $\psi_\cD$ and $\psi_{\cD,\Lambda}$ will produce the same sandpile configuration on $D_k$. Therefore to bound the difference of the
measures on any cylinder event $E$ defined on $D_k$ it suffices to bound the probability that the descendants in the spanning trees differ, or
the trees differ on that set of descendants.
\begin{equation}
\label{e:final-bound-d>2}
\begin{split}
  |\nu(E) - \nu_\Lambda(E)|
  &\le \P ( E_{\Z^d} \Delta E_\Lambda ) \\
  &\le \P \left( \text{$W_k \not= W_{k,\Lambda}$ or $T\vert_{\sfE^*_{W,k}} \not= T_\Lambda\vert_{\sfE^*_{W,k}}$} \right) \\
  &\le C_d \frac{k^{d-2} n^2}{(N-n)^{d-2}} + C_d k^{d-1} n^{\frac{2-d}{2d}}.
\end{split}
\end{equation} We now optimise the choice of $n$. We may assume $N \geq 2n$, in
which case $(N-n)^{d-2} \ge c_d N^{d-2}$.

When $d \ge 5$, we take $n = \frac{1}{2}N$, which gives the bound
$C_d k^{d-1} N^\frac{2-d}{2d}$.

When $d = 4$, the two terms in the right hand side of \eqref{e:final-bound-d>2}
are of the same order if $n = k^{4/9} N^{8/9}$. This gives the bound
$C k^{26/9} N^{-2/9}$.

When $d = 3$, we take $n = k^{6/13} N^{6/13}$. This yields the bound
$C k^{17/9} N^{-1/9}$. \end{proof}

\section{Rate of convergence in $\mathbb{Z}^2$.}
\label{sec:d=2}

In this section we bound the rate of convergence on $\Z^2$ in Theorem \ref{thrm:d=2}
below. As was the case for $d\geq 3$, the result will follow directly from the bijections and a bound on the probability that, in a suitable coupling, the descendants of $D_k$ in $\Z^2$ differ from those in $\Lambda$. This bound is given in Proposition \ref{tree d=2}. Due to recurrence, we cannot use Wilson's method rooted at infinity, so the
construction of the coupling is more involved. Write $G = (\Lambda \cup \{ s \},\sfE_\Lambda)$
for the graph on which the sandpile is defined. Recall that given a cylinder event $E$ determined by the sandpile heights in $D_k$, we write
$E_{\Z^2} = \{ \psi_\cD(T) \in E \}$ and $E_\Lambda = \{ \psi_{\cD,\Lambda}(T_\Lambda) \in E \}$, where $T$ is a sample from $\WSF$ and $T_\Lambda$ is a sample from $\UST_G$.

\begin{theorem}
\label{thrm:d=2}
Let $E$ be a cylinder event determined by the sandpile heights in $D_k$, and
let $\Lambda \subset \Z^2$ be a finite set. Let $N$ be the largest integer such that $D_N \subset \Lambda$. Given $\eps > 0$, there is a constant $C = C(\eps) > 0$ and a coupling $\P = \P_{\Lambda,k,\eps}$ of
$T$ and $T_\Lambda$, such that in this coupling we have
\eqnst
{ \left| \nu(E) - \nu_\Lambda(E) \right|
  \le \P ( E_{\Z^2} \Delta E_\Lambda )
  \le C \frac{k^{5/32}}{N^{1/16-\varepsilon}}. }
\end{theorem}

We will write $W_k$, respectively $W_{k,\Lambda}$,
for the set of descendants of $D_k$ in $T$, respectively $T_\Lambda$.
Then Theorem \ref{thrm:d=2} follows immediately from the following proposition.

\begin{proposition}
\label{tree d=2}
For any $\eps > 0$ there exists $C = C(\eps) > 0$ and a coupling $\P = \P_{\Lambda,k,\eps}$ of $T$ and $T_\Lambda$ such that in
this coupling \begin{align*}
  \P \left( \text{$W_k \neq W_{k,\Lambda}$ or
     $T$ and $T_\Lambda$ differ on some edge touching $W_k$} \right)
  \leq C \frac{k^{5/32}}{N^{1/16-\varepsilon}}.
\end{align*}
\end{proposition}

The coupling will be achieved by passing to the planar dual graphs.
The idea is to construct paths in the dual tree that together surround $D_k$ in
such a way that all descendants of $D_k$ are necessarily in the
interior of the region defined by the paths. Then it will be sufficient
to couple the dual trees in the interior of that region.

Let $G^* = (\Lambda^*,\sfE_\Lambda^*)$ denote the planar dual of $G$. The vertex set $\Lambda^*$ is naturally identified with a subset of the dual lattice $(\Z^2)^* = \Z^2 + (1/2,1/2)$. The planar graph $G^*$
has one unbounded face: the face corresponding to the sink $s$ via duality.
The dual spanning tree $T_\Lambda^*$ is defined on $G^*$, by including a dual edge $e^*$ in $T_\Lambda^*$ if and only if the corresponding edge $e$ is not in $T_\Lambda$. Then $T_\Lambda^*$ is a sample from $\UST_{G^*}$ (i.e.~with
free boundary conditions). It is well known that as $\Lambda \uparrow \Z^2$, the measure $\UST_{G^*}$ converges weakly to the free spanning forest measure $\FSF$, which for $\Z^2$ coincides with $\WSF$ \cite{Pem(91),LPbook}.
Let $T^*$ denote a sample from this measure on the graph $(\Z^2)^*$.
We refer to paths in $\Z^2$ as \emph{primal} paths, and paths in
$(\Z^2)^*$ as \emph{dual} paths. Let $o^*$ be the dual vertex $o + (1/2,1/2) \in (\Z^2)^*$,
where $o$ is the origin in $\Z^2$. For any $m \ge 0$ we define the balls 
in the dual graph:
\eqnst
{ D^*_m
  := \{ w \in (\Z^2)^* : |w - o^*| \le m \}. }

The construction of the coupling is broken down into a sequence of steps, and
the required estimates stated as lemmas. We collect the estimates at the
end and prove Proposition \ref{tree d=2}.
The integers $\ell \ge 1$ and $k < n < r < R < N$ will be parameters that we choose at the end to optimize the bound.

\medbreak

\emph{Step 1. Coupling the backbones inside $D^*_r$.}
We will need to work with fixed ``backbones'' in our trees. Since $T^*$ has one end $\WSF$-a.s., 
there is a unique infinite path $\gamma^*$ in 
$T^*$ that starts at $o^*$. We call $\gamma^*$ the \emph{backbone} of $T^*$.
The free spanning tree on $\Lambda^*$ does not have a unique backbone (there are typically several paths from $o^*$ to the boundary of $\Lambda^*$). Therefore, we will first work with the \emph{wired} boundary condition in the \emph{dual graph}, i.e.~we consider the graph $\widetilde{G}^* = (\Lambda^* \cup \{ s^* \}, \widetilde{\sfE}_\Lambda^*)$
obtained by connecting each vertex in $\Lambda^*$ to $s^*$ by as many edges
as it needs, for its degree to be $4$.
Then we will compare $\UST_{\widetilde{G}^*}$ to $\UST_{G^*}$
using the well known monotone coupling between them \cite{Pem(91),LPbook}.
Let $\widetilde{T}_\Lambda^*$ denote a sample from $\UST_{\widetilde{G}^*}$.
Let $\gamma^*_\Lambda$ denote the unique path between $o^*$ and $s^*$ in $\widetilde{T}_\Lambda^*$.
We call $\gamma^*_\Lambda$ the \emph{backbone} of $\widetilde{T}_\Lambda^*$.

We fix a coupling between $\gamma^*$ and $\gamma^*_\Lambda$ that maximizes the probability that their first $\ell$ steps are identical. The next lemma collects some LERW estimates from the literature that we use to
estimate the probability that the restrictions of $\gamma^*_\Lambda$ and $\gamma^*$ to the ball $D^*_r$ differ from each other.

\begin{lemma}
\label{results}
(i) For $l < \sqrt{N}$, we have
\eqnst
{ \P(\text{first $\ell$ steps of $\gamma^*$ and $\gamma^*_\Lambda$ are not identical})
  \leq C \frac{l^2}{N} \ln\left(\frac{N}{l}\right). }
(ii) If $R > 4r$, we have
\eqnst
{ \P(\text{$\gamma^*_\Lambda$ returns to $D^*_r$ after its first exit from $D^*_R$})
  \le C \frac{r}{R} }
and
\eqnst
{ \P(\text{$\gamma^*$ returns to $D^*_r$ after its first exit from $D^*_R$})
  \le C \frac{r}{R}. }
(iii) We have
\eqnst
{ \E[\text{number of steps of $\gamma^*$ until first exit from $D^*_R$}]
  = R^{\frac{5}{4}+o(1)} \quad \text{as $R \to \infty$.} }
(iv) For all $\lambda, \eps > 0$, $N > 4R$ we have that there exist $C(\eps), C_1, C_2 > 0$
such that \eqnst{
  \P(\text{number of steps of $\gamma^*_\Lambda$ until first exit from $D^*_R >\lambda C(\eps) R^{\frac{5}{4}+\eps}$})
  \leq C_1 e^{-C_2\lambda}. }
\end{lemma}

\begin{remark}
Note that in contrast with \cite[Proposition 11.3.1]{LLbook}, the above bounds give us
power law (rather than logarithmic) control on the errors, since we are free to
discard a collection of ``bad paths'' in $D^*_r$ of small probability on which convergence to the infinite LERW would be much slower.
\end{remark}

\begin{proof}[Proof of Lemma \ref{results}.]
(i) The statement follows from \cite[Proposition 7.4.2]{Lawler}. Note that although the exact statement is not present in the reference, it immediately follows from the proof presented there.

(ii) This is \cite[Lemma 2.4]{Barlow(11)}.

(iii) This result was first shown by Kenyon \cite{Kenyon(00)} (stated there
in the upper half plane). It also follows by combining \cite[Proposition 6.2(2)]{Barlow(10)}
and \cite[Theorem 5.7]{Masson(09)}.

(iv) This follows from \cite[Corollary 3.4]{Barlow(10)}, \cite[Theorem 5.8(4)]{Barlow(10)} and part (iii).
\end{proof}

The next lemma puts the above estimates together and bounds the probability that the restrictions of $\gamma^*_\Lambda$ and $\gamma^*$ to the ball $D^*_r$ are not
identical.

\begin{lemma}
\label{backbone constant}
\begin{align*}
  &\P( \gamma^*_\Lambda \cap D^*_r \not= \gamma^* \cap D^*_r )
  \leq C \frac{\lambda^2 C(\eps)^2 R^{\frac{5}{2}+2\eps}}{N} \ln\left(\frac{N}{\lambda C(\eps) R^{5/4+\eps}}\right)
          + C_1 \exp (-C_2 \lambda) + 2 C \frac{r}{R}
\end{align*}
\end{lemma}

\begin{proof}
Let $F_1$ be the event that the first $\ell$ steps of $\gamma^*$ and $\gamma^*_\Lambda$ coincide, the event maximized by our choice of coupling. We therefore need to choose $\ell$ appropriately to get the desired result.

Let $F_2$ be the event that the length of $\gamma_\Lambda^*$ until first exit of $D^*_R$ is less than $\ell$.

Let $F_3$ be the event that neither $\gamma_\Lambda^*$ nor $\gamma^*$ return to $D^*_r$ after their first exits from $D^*_R$.

On the event $F_2 \cap F_3$, we have that the first $\ell$ steps of $\gamma_\Lambda^*$ includes
$\gamma^*_\Lambda \cap D^*_r$. If $F_1$ also occurs, then we have $\gamma^*_\Lambda \cap D^*_r = \gamma^* \cap D^*_r$.
We choose $\ell = \lambda C(\eps) R^{5/4+\eps}$. By Lemma \ref{results}(i),(iv),(ii) we have
\eqnsplst
{ \P (\gamma^*_\Lambda \cap D^*_r \not= \gamma^* \cap D^*_r)
  &\le \P(F_1^c) + \P(F_2^c) + \P(F_3^c) \\
  &\le C \frac{\lambda^2 C(\eps)^2 R^{\frac{5}{2}+2\eps}}{N} \ln\left(\frac{N}{\lambda C(\eps) R^{5/4+\eps}}\right)
       + C_1 \exp (-C_2 \lambda)
       + 2 C \frac{r}{R}. }
\end{proof}

\emph{Step 2. Constructing the dual paths that surround $D^*_k$.}
On the event $\gamma^*_\Lambda \cap D^*_r \not= \gamma^* \cap D^*_r$, we extend the coupling of
$\gamma^*_\Lambda$ and $\gamma^*$ to a coupling of $\WSF$ and $\UST_{\widetilde{G}^*_\Lambda}$
in an arbitrary way. (For example: make them conditionally independent given the backbones.)
On the event $\gamma^*_\Lambda \cap D^*_r = \gamma^* \cap D^*_r$, we extend the coupling via Wilson's stacks of arrows construction. For each $x \in D^*_r \setminus \gamma^*$, we assign
identical stacks for the constructions in $\Lambda^*$ and $(\Z^2)^*$, respectively. For all other vertices, the stacks in $\Lambda^*$ are assigned independently from those in $(\Z^2)^*$.
This defines a coupling of $\WSF$ and $\UST_{\widetilde{G}^*}$ on $(\Z^2)^*$.

We now construct the required set of dual paths. Write $\gamma^*_r$ for the portion of
$\gamma^*$ up to its first exit from $D^*_r$. By a \emph{block}, we mean a set $U$ of dual edges with the properties:\\
(i) $U \subset D^*_n \setminus D^*_k$;\\
(ii) $U \cup \gamma^*_r$ is a connected set of edges;\\
(iii) the set of vertices of $U \cup \gamma^*_r$ disconnects $D^*_k$ from $(D^*_n)^c$.

\begin{lemma} 
\label{lem:almost loop}
Suppose that $r > 4n > 16k$ and $\gamma^*_\Lambda \cap D^*_r = \gamma^* \cap D^*_r$ . There exists $C > 0$ such that \eqnst
{ \P \left( \parbox{8.5cm}{we can pop a set of coloured cycles contained in
     $D^*_n \setminus D^*_k$ so that the arrows revealed contain a block} \right)
  \ge 1 - C \left(\frac{k}{n}\right)^{1/4} - C \frac{n}{r}. }
\end{lemma}

\begin{proof}
Due to Lemma \ref{results}(ii), we have $\P ( \gamma^* \cap D^*_n \not= \gamma^*_r \cap D^*_n ) \le C (n/r)$. Henceforth assume that we are on the event when $\gamma^* \cap D^*_n = \gamma^*_r \cap D^*_n$.

We start with a minor adaptation of the argument of \cite[Lemma 6.1]{ABNW(99)}.
Let $v \in (\Z^2)^*$ be a vertex at distance $\sqrt{k n}$ from $o^*$, and let 
$\{ S(n) \}_{n \ge 0}$ be simple random walk starting at $v$. Let $\tau$ be the 
first time when either $S$ exits $D^*_n \setminus D^*_k$, or when the loop-erasure 
of $S$ has made a non-contractible loop around $D^*_k$. Let us use the sequence 
$S(1), S(2), \dots, S(\tau)$ as our successive choices in Wilson's algorithm, 
where $\gamma^*_r$ is already part of the tree to be constructed. That is, 
whenever a random walk step is to be made, we use the next step of $S$ for the 
random walk step, and whenever a new vertex is to be chosen in the algorithm, 
we use the next vertex visited by $S$ as the new vertex.

We claim that on the event $S[0,\tau] \subset D^*_n \setminus D^*_k$ the
set of edges $U$ that we have added to the tree is a block. Indeed, condition (i)
holds because the walk never left $D^*_n \setminus D^*_k$. Also, observe that the 
set of vertices of $\LE(S[0,\tau))$ do not get erased, and hence condition (iii)
holds. Finally, condition (ii) holds, because each piece of the tree we create 
gets joined to $\gamma^*_r$ (here is where we use that 
$\gamma^* \cap D^*_n = \gamma^*_r \cap D^*_n$). Note that since $S(\tau-1)$ does not
get erased, the last piece is also joined. This proves the claim. Interpreting 
the construction in terms of stacks of arrows, we see that the probability of 
the event in part (i) is at least the probability that
$S[0,\tau] \subset D^*_n \setminus D^*_k$.

The probability that a non-contractible loop is created could be bounded by
$\ge 1 - C (k/n)^\zeta$ with some $\zeta, C > 0$, by ideas similar 
to \cite[Exercise 3.3]{LLbook}, showing the statement (i) with $\zeta$ in place of $1/4$.
In order to get the explicit exponent $1/4$, we combine the argument with an idea 
that was inspired by \cite{Benj(00)}.

Again we are going to start with $\gamma^*_r$ as our initial tree.
Choose a subpath $\gamma^*_{k,n}$ of $\gamma^*_r$ that forms a crossing from
$D^*_k$ to $(D^*_n)^c$. Write $H_\rho$ for the circle of radius $\rho$ centred at $o^*$. Define the annulus:
\eqnst
{ A_{k,n}
  = \{ z \in \R^2 : k+1 < |z - o^*| < n-1 \}. }
Choose a point $Q \in \gamma^*_{k,n} \cap H_{\sqrt{k n}}$, and let $\alpha_0 = H_{\sqrt{k n}} \setminus \{ Q \}$. Let $P_1$ be the mid-point of $\alpha_0$, and let $v_1$ be a vertex of $(\Z^2)^*$ closest to $P_1$. Run a random walk $S^{(1)}$ from $v_1$ to $\gamma^*_{k,n}$, and add edges to the tree in the same way as we did with $S$. Let $\pi_1$ be the set of edges added.
Note that $\pi_1$ is not necessarily a connected set of edges, however,
$\gamma^*_r \cup \pi_1$ is.
From the two subarcs of $\alpha_0$ defined by $P_1$, throw away the one that is on the same side of $\gamma^*_{k,n}$ as where $\pi_1$ hit, and let us call the other arc $\alpha_1$. On the event when $\{ S^{(1)} \} \subset A_{k,n}$, the arc $\alpha_1$ has the property that any dual lattice path from $H_k$ to $H_n$ that is vertex-disjoint from $\gamma^*_r \cup \pi_1$ has to intersect $\alpha_1$.

Continue inductively in the following way. Suppose that for some $i \ge 1$ the
arc $\alpha_i$ and the sets of edges $\pi_1, \dots, \pi_i$ have been defined. 
Let $P_{i+1}$ be the mid-point of $\alpha_i$ and let $v_{i+1}$ be the vertex of $(\Z^2)^*$ closest to $P_{i+1}$. Run a random walk $S^{(i+1)}$ from $v_{i+1}$
to $\gamma^*_{k,n} \cup \pi_1 \cup \dots \cup \pi_i$, and let $\pi_{i+1}$ be the set of edges that get added to the tree. From the two subarcs
of $\alpha_i$, throw away the one that is on the same side of $\gamma^*_r$ as where $\pi_{i+1}$ hit, and call
the other one $\alpha_{i+1}$. On the event when $\{ S^{(i+1)} \} \subset A_{k,n}$,
the arc $\alpha_{i+1}$ has the property that any dual lattice path from $H_k$ to $H_n$ that is vertex-disjoint from $\gamma^*_r \cup \pi_1 \cup \dots \cup \pi_{i+1}$
has to intersect $\alpha_{i+1}$.

The construction is well defined until a time when the length of the arc $\alpha_i$ becomes of order $1$. Stop the construction the first time when $\diam(\alpha_i) < 10$, say. We can select further vertices $v_{i+1}, \dots, v_{i+K}$ (with $K$ a fixed constant, say, $K = \lceil 10 \sqrt{2} + 4 \rceil$) such that if we start further random walks at these vertices, then $\gamma^* \cup \pi_1 \cup \dots \cup \pi_{i+K}$ contains a block.
An example of the start of this construction is shown in Figure \ref{fig:block}.

\psfrag{gamma^*_r}{$\gamma^*_r$}
\psfrag{gamma^*_k,n}{$\gamma^*_{k,n}$}
\psfrag{r}{$r$}
\psfrag{n}{$n$}
\psfrag{k}{$k$}
\psfrag{v_1}{$v_1$}
\psfrag{v_2}{$v_2$}
\psfrag{v_3}{$v_3$}
\psfrag{v_4}{$v_4$}
\psfrag{v_5}{$v_5$}

\begin{figure}
\centering
\includegraphics[scale=.6]{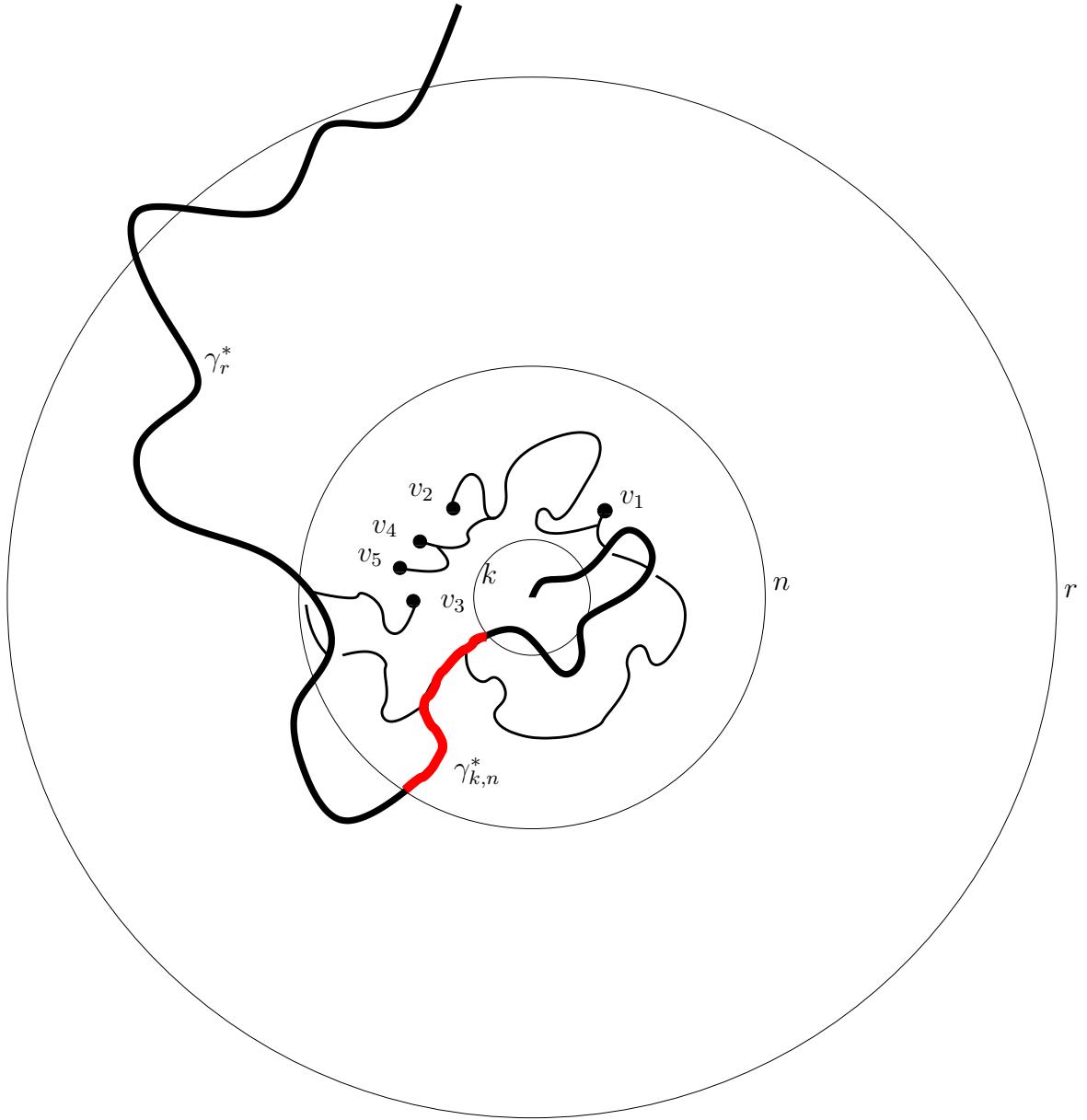}
\caption{An example of the construction of a block. The thick line is $\gamma^*_r$, and the red piece is $\gamma^*_{k,n}$. LERWs were started successively at $v_1$, $v_2$, etc. Note the gaps between pieces in some of the LERWs, where an intersection with $\gamma^*_r \setminus \gamma^*_{k,n}$ has occurred.}
\label{fig:block}
\end{figure}

It remains to bound the probability that the walks $S^{(1)}, S^{(2)}, \dots$ all remain inside 
$D^*_n \setminus D^*_k$. The $i$-th walk $S^{(i)}$ starts at distance $O(2^{-i} \sqrt{k n})$ 
from the current tree $\mathfrak{T}_{i-1} := \gamma^*_{k,n} \cup \pi_1 \cup \dots \cup \pi_{i-1}$.
If it were to leave $D^*_n \setminus D^*_k$ without hitting $\mathfrak{T}_{i-1}$, it would 
first have to leave the ball 
\eqnst
{ B^*(v_i; (1/4) \sqrt{kn})
  := \{ w \in (\Z^2)^* : |w - v_i| \le (1/4) \sqrt{kn} \}. } 
without hitting $\mathfrak{T}_{i-1}$. 
Using Beurling's estimate \cite[Section 6.8]{LLbook}, the probability of this
is at most $C (2^{-i} \sqrt{kn} / \sqrt{kn} )^{1/2}$.
Regardless of where the walk exits $B^*(v_i; (1/4) \sqrt{kn})$, the exit point
$z^*_i$ is still at distance $\asymp \sqrt{k n}$ from $o^*$.
It follows, again using Beurling's estimate, that the probability that the 
continuation of the walk from $z^*_i$ exits $D^*_n$ without hitting 
$\mathfrak{T}_{i-1}$ is at most $C (\sqrt{kn} / n )^{1/2}$. Similarly,
together with a time-reversal argument, the probability that the 
walk started at $z^*_i$ hits $D^*_k$ before hitting $\mathfrak{T}_{i-1}$ is
at most $C (k / \sqrt{kn} )^{1/2}$. Combining these three estimates we get the bound
\eqnsplst
{ \P ( \{ S^{(i)} \} \not\subset D^*_n \setminus D^*_k \,|\, \mathfrak{T}_{i-1} )
  &\leq C \left(\frac{2^{-i} \sqrt{k n}}{\sqrt{k n}}\right)^{\frac{1}{2}} \times
       \left[ \left(\frac{\sqrt{k n}}{n}\right)^{\frac{1}{2}}
       + \left(\frac{k}{\sqrt{k n}}\right)^{\frac{1}{2}} \right] \\
  &= O \left( 2^{-i/2} \left(\frac{k}{n}\right)^{\frac{1}{4}} \right). }
Summing over $i$ we get the claimed bound $O((k/n)^{1/4})$.
\end{proof}

\medbreak

\emph{Step 3. Coupling the set of descendants.}
We now complete the definition of the coupling of $T_\Lambda$ and $T$.
Fix a monotone coupling between $\widetilde{T}^*_\Lambda$ and $T^*_\Lambda$, such that $e^* \in \widetilde{T}^*_\Lambda \cap \sfE^*_\Lambda$ implies
$e^* \in T^*_\Lambda$; see \cite[Chapter 10]{LPbook}. Define $T_\Lambda$ and $T$ as the dual trees of $T^*_\Lambda$ and $T^*$.
This completes the definition of required coupling $\P_{\Lambda,k,\eps}$.

\begin{lemma}
\label{lem:contained} \ \\
(i) When $\widetilde{T}^*_\Lambda$ contains a block, we have 
$W_{k,\Lambda} \subset D_r$. \\
(ii) When $T^*$ contains a block, we have $W_k \subset D_r$. \\
(iii) When the event in Lemma \ref{lem:almost loop} occurs, we have
$W_{k,\Lambda} = W_k \subset D_r$ and $T$ and $T_\Lambda$
agree on the set of edges with at least one endvertex in $W_k$.
\end{lemma}

\begin{proof}
(i) Since $T^*_\Lambda$ is stochastically larger than $\widetilde{T}^*_\Lambda$,
the edges in the block are also present in $T^*_\Lambda$. Since the the union 
of the block with $\gamma^*_r$ is connected, any two dual vertices in the interior 
of the block are connected by a path in $\widetilde{T}^*_\Lambda$. Hence no new 
edges are added in the interior of the block when passing from 
$\widetilde{T}^*_\Lambda$ to $T^*_\Lambda$.

Suppose that $D_k$ had a descendant $v \in D_r^c$ in $T_\Lambda$. Then there would 
be a primal path $\beta$ starting at $v$ that visits $D_k$ and ends outside $D_r$. 
Since the block surrounds $D_k$, this would contradict the connectivity of the block 
(as a set of edges). 

(ii) The same argument as in the previous paragraph applies here.

(iii) Since we are using the same stacks of arrows in $D^*_r \setminus \gamma^*_r$,
the same block exists in $\Lambda^*$ and in $(\Z^2)^*$, and the trees
coincide in the interior of the region defined by the block. Therefore, the trees 
$T_\Lambda$ and $T$ also coincide in this region. By parts (i) and (ii), 
the set of descendants are contained in this region
and are equal in $T_\Lambda$ and $T$.
\end{proof}

\medbreak

\begin{proof}[Proof of Proposition \ref{tree d=2}]
By Lemma \ref{lem:contained} we have $W_{k,\Lambda}=W_k$ if the event in Lemma \ref{lem:almost loop} occurred which in turn assumed that the event in Lemma \ref{backbone constant} did not occur. Therefore we have
\begin{align*}
 &\P \left( \text{$W_k \neq W_{k,\Lambda}$ or $T$ and $T_\Lambda$ differ on some edge
       with an endvertex in $W_k$} \right) \\
 &\qquad \leq C \frac{\lambda^2 C(\eps)^2 R^{\frac{5}{2}+2\eps}}{N} \ln\left(\frac{N}{\lambda C(\eps) R^{5/4+\eps}}\right)
          + C_1 \exp (-C_2 \lambda) +  C \frac{r}{R}
          +  C \left(\frac{k}{n}\right)^{1/4} + C \frac{n}{r}. \end{align*}
We can now optimise our choice of parameters by taking $n=(kr^4)^{1/5}$, $r=(R^5k)^{1/6}$, $R=(k N^{6})^{1/16}$ and choose $\lambda$ such that $\lambda^2 R^{2\eps}=N^\eps$.
\end{proof}

\medbreak

\textbf{Acknowledgements.} We are grateful to an anonymous referee for useful suggestions.

\medbreak

\end{document}